\documentclass[12pt,reqno]{amsart}

\usepackage[all,cmtip]{xy}
\usepackage{tikz-cd}
\usepackage{enumitem}
\usepackage{parskip}
\usepackage{nicematrix}
\usepackage{amsmath}
\usepackage{amssymb}
\usepackage{mathtools}
\usepackage{hyperref}
\usepackage[capitalise]{cleveref}
\numberwithin{equation}{section}

\newtheorem{theorem}{Theorem}[section]
\newtheorem{corollary}[theorem]{Corollary}
\newtheorem{lemma}[theorem]{Lemma}

\theoremstyle{definition}

\newtheorem{example}[theorem]{Example}

\newcommand{\GmodX}[3]{#1 \backslash\mkern-6mu\backslash #2^{ss}_{#1}(#3)}

\begin{document}	
	\title{GIT quotient of minimal dimensional Schubert variety modulo a subtorus }	
	\author{Arkadev Ghosh }
		\address{Chennai Mathematical Institute, Plot H1, SIPCOT IT Park, Siruseri, Kelambakkam, 603103, India.}
		\email{arkadev@cmi.ac.in}	
			\author{S. S. Kannan}
		\address{Chennai Mathematical Institute, Plot H1, SIPCOT IT Park, Siruseri, Kelambakkam, 603103, India.}	
		\email{kannan@cmi.ac.in}	
	\maketitle 		

\begin{abstract}
Let $G=PSL(n,\mathbb{C})$. Let $T$ be a maximal torus of $G$. Let $\omega_{r}$ denote the $r^{th}$ fundamental weight.  Let $\mathcal{L}(n\omega_{r})$ denote the line bundle on the Grassmannian $G_{r,n}$ associated to the character $n\omega_{r}$ of $T$.
In \cite{sskannan}, it is proved that there is a unique minimal dimensional Schubert variety $X(w_{r,n})$ in $G_{r,n}$ admitting semistable points for the $T$-linearized ample line bundle $\mathcal{L}(n\omega_{r})$. Assume that $n=rq+1$, where $r,q\in\mathbb{N}$ and $q\geq 2$. In this paper, we study the GIT quotient of $X(w_{r,n})$ modulo a subtorus $T_{J_{r}}$ of $T$ generated by the one parameter subgroups of $T$ corresponding to the peaks of $w_{r,n}$ (see \cite[Definition 4.6]{Perrin}). We prove that the GIT quotient of $X(w_{r,n})$ modulo $T_{J_{r}}$ is isomorphic to the total space of the $r^{th}$ stage of an iterated projective space bundle over $\mathbb{P}^{q-1}$.
\end{abstract}		
\section{introduction} Let $r,q\in\mathbb{N}$ with $q\geq 2$ and $n=rq+1$. Let $G=PSL(n,\mathbb{C})$.
Let $T$ be a maximal torus of $G$ and $B$ be a Borel subgroup of $G$ containing $T$. Let $W=N_{G}(T)/T$ denote the Weyl group of $G$ with respect to $T$. The set of roots of $G$ with respect to $T$ will be denoted by $R$. Let $R^{+}\subseteq R$ be the set of positive roots with respect to $(B,T)$. Let $S=\{\alpha_{1},\ldots,\alpha_{n-1}\}\subseteq R^{+}$ be the set of simple roots of $G$ with respect to $B$. The simple reflection in $W$ corresponding to $\alpha_{i}$ is denoted by $s_{i}$. Let $\{\omega_{r}:1\leq r\leq n-1\}$ denote the fundamental weights associated to $S$. Let $\{\lambda_{1},\ldots,\lambda_{n-1}\}$ denote the set of one parameter subgroups of $T$ dual to $S$. Let $P_{S\setminus\{\alpha_{r}\}}$ be the standard maximal parabolic subgroup of $G$ corresponding to $S\setminus\{\alpha_{r}\}$. Then the Grassmannian $G_{r,n}$ is isomorphic to $G/P_{S\setminus\{\alpha_{r}\}}$.

Consider the action of $T$ on the Grassmannain $G_{r,n}$. In \cite{HK}, Hausmann and Knutson identified the GIT quotient of Grassmannian $G_{2,n}$ by $T$ with the moduli space of polygons in $\mathbb{R}^{3}$. They  showed that GIT quotient of $G_{2,n}$ by $T$ can be realized as the GIT quotient of the $n$-fold product of projective lines by the diagonal action of $PSL(2,\mathbb{C})$. More generally, using Gel'fand-Macpherson correspondence GIT quotient of $G_{r,n}$ by $T$ can be identified with the GIT quotient of $(\mathbb{P}^{r-1})^{n}$ by the diagonal action of $PSL(r,\mathbb{C})$.

In \cite{Sko}, Skorobogatov gave a combinatorial description of the set of semistable points in $G_{r,n}$ for the $T$-linearized line bundle $\mathcal{L}(n\omega_{r})$. As a corollary, he showed that when $r$ and $n$ are coprime, semistability is same as stability. In \cite{TorusquotientI} and \cite{TorusquotientII}, second named author characterized the parabolic subgroups $Q$ of a simple algebraic group $H$ containing a fixed Borel subgroup, for which there exists an ample line bundle $\mathcal{L}$ on $H/Q$
such that $(H/Q)^{ss}_{T}(\mathcal{L})=(H/Q)^{s}_{T}(\mathcal{L})$. In particular, when $H=PSL(n,\mathbb{C})$ and $Q=P_{S\setminus\{\alpha_{r}\}}$, he showed that $(G_{r,n})^{ss}_{T}(\mathcal{L}(n\omega_{r}))=(G_{r,n})^{s}_{T}(\mathcal{L}(n\omega_{r}))$ if and only if $r$ and $n$ are coprime.

 In \cite{sskannan} second named author and Sardar proved that there is a unique minimal dimensional Schubert variety $X(w_{r,n})$ in $G_{r,n}$ admitting semistable points for the $T$-linearized line bundle $\mathcal{L}(n\omega_{r})$ and gave a combinatorial description of $w_{r,n}$. In \cite{KP}, Kannan and Pattanayak extended the results to the simple algebraic groups of type $B$, $C$ and $D$. 
 In \cite{KB}, S.Bakshi, S.S.Kannan and K.V. Subrahmanyam proved that $\GmodX{T}{X(w_{r,n})}{\mathcal{L}(n\omega_{r})}$ is smooth and  $\GmodX{T}{X(w_{3,7})}{\mathcal{L}(7\omega_{3})}$ is a rational normal scroll.
 
  Using that description of $w_{r,n}$ in \cite{sskannan}, we have a reduced expression 
\begin{align*}
 w_{r,n}=(s_{q}\cdots s_{1})(s_{2q}\cdots s_{2})\cdots(s_{rq}\cdots s_{r}).
\end{align*} Note that $\{\alpha_{iq}:1\leq i\leq r\}$ are the peaks of $w_{r,n}$(see \cite[Definition 4.6]{Perrin}). Let $T_{J_{r}}$ be the subgroup of $T$ generated by $\lambda_{iq}(\mathbb{G}_{m})$ ($1\leq i\leq r$). 

We observe that $X(w_{r,n})$ is also the minimal dimensional Schubert variety in $G_{r,n}$ admitting semistable points for the $T_{J_{r}}$-linearized line bundle $\mathcal{L}(n\omega_{r})$ (see \cref{sction:combinatorial resul:2}). So, it is interesting to study the GIT quotient $\GmodX{T_{J_{r}}}{X(w_{r,n})}{\mathcal{L}(n\omega_{r})}$. In this direction, we prove the following results.

Let $Y_{r}:=\GmodX{T_{J_{r}}}{X(w_{r,n})}{\mathcal{L}(n\omega_{r})}$. 
For $1\leq k\leq r-1$,  let $X(w_{k,kq+1})$ denote the unique minimal dimensional Schubert variety in $G_{k,kq+1}$ admitting semistable points for the $T_{k}$-linearized line bundle $\mathcal{L}((kq+1)\omega_{k(k)})$  (see \cref{section 6.1} for notations). Define\begin{align*}
	Y_{k}:=\GmodX{T_{J_{k}}}{X(w_{k,kq+1})}{\mathcal{L}((kq+1)\omega_{k(k)})}.
\end{align*}
\begin{theorem}$($See \cref{main result 1} $)$ We have a morphism $\phi_{r}:Y_{r}\longrightarrow Y_{r-1}$ such that 
$Y_{r}$ is a $\mathbb{P}^{r(q-1)}$-bundle over $Y_{r-1}$.
\end{theorem}
Since $Y_{r-1}$ is smooth (see \cref{semistability=stability}), $Y_{r}$ is isomorphic to a projective space bundle $\mathbb{P}(\mathcal{E})$ of a vector bundle $\mathcal{E}$ on $Y_{r-1}$(see \cite[Ch 2, Exercise 7.10(c)]{Ha}). We identify the vector bundle $\mathcal{E}$(up to a twist by line bundle). We find a trivializing open cover of $Y_{r-1}$ for $\mathcal{E}$ such that the images of the transition functions lie inside the maximal torus of $GL(r(q-1)+1,\mathbb{C})$ consisting of diagonal matrices. From this we derive that the vector bundle $\mathcal{E}$ splits into direct sum of line bundles. More precisely, we have the following.
\begin{theorem} $($See \cref{section:projective bundle:th 1}$)$
	We have line bundles $\mathcal{L}_{0},\ldots,\mathcal{L}_{r-1}$ on $Y_{r-1}$ and an isomorphism $\Psi:Y_{r}\longrightarrow \mathbb{P}(\mathcal{E})$ over $Y_{r-1}$, where $\mathcal{E}=\mathcal{L}_{0}^{\oplus q}\oplus(\bigoplus_{j=1}^{r-1} \mathcal{L}_{j}^{\oplus q-1})$.
\end{theorem}
\begin{theorem}$($See \cref{section:projective bundle:cor 1}$)$
	For $1\leq k\leq r$, we have morphism $\phi_{k}:Y_{k}\longrightarrow Y_{k-1}$ such that 
	\begin{align*}
		Y_{r}\xrightarrow{\text{$\phi_{r}$}}Y_{r-1}\xrightarrow{\text{$\phi_{r-1}$}}\cdots Y_{2}\xrightarrow{\text{$\phi_{2}$}}Y_{1}\xrightarrow{\text{$\phi_{1}$}}Y_{0}=\{pt\}
	\end{align*}
	is a generalized Bott tower with stage $r$ (see \cref{subsection 2.2} for the Definition of generalized Bott tower).
	\end{theorem}
\subsection{Organization}The organization of the article is as follows. In section [2], we recall some notations and preliminaries on algebraic groups, geometric invariant theory, projective bundle and generalized Bott towers. In section [3], we prove that $X(w_{r,n})$ is the unique minimal dimensional Schubert variety admitting semistable points in $G_{r,n}$ for the $T_{J_{r}}$-linearized line bundle $\mathcal{L}(n\omega_{r})$. In section [4], we find a description of the semistable locus $X(w_{r,n})^{ss}_{T_{J_{r}}}(\mathcal{L}(n\omega_{r}))$. In section [5], we prove that  $Y_{r}$ is a $\mathbb{P}^{r(q-1)}$-bundle over $Y_{r-1}$. In section [6], we prove that $Y_{r}$ is isomorphic to the total space of the $r^{th}$ stage of an iterated projective space bundles over the projective space $\mathbb{P}^{q-1}$.
\section{notations and preliminaries}
In this section, we set up some notations and preliminaries. Let $G$, $B$, $T$, $R$, $S$ and $W$ be as in the introduction. We refer to \cite{Hum1},\cite{Hum2}, and \cite{Jan} for preliminaries and notation for algebraic groups and Lie algebras. We refer to \cite{GIT}, and \cite{NEWSTEAD} for preliminaries and notations for Geometric invariant theory.\\
The $\mathfrak{g}$ be the Lie algebra of $G$. Let $\mathfrak{b}\subseteq\mathfrak{g}$ and $\mathfrak{h}\subseteq\mathfrak{b}$ be the Lie algebras of $B$ and $T$ respectively.\\
Let $\hat{G}=SL(n,\mathbb{C})$ and $\pi:\hat{G}\longrightarrow G$ be the simply connected covering of $G$. Let $\hat{T}=\pi^{-1}(T)$ and $\hat{B}=\pi^{-1}(B)$.

For a subset $I$ of $S$, we denote the parabolic subgroup of $G$ generated by $B$ and $\{n_{\alpha}:\alpha\in I\}$ by $P_{I}$, where $n_{\alpha}$ is a representative of $s_{\alpha}$ in $N_{G}(T)$. Note that every standard parabolic subgroup of $G$ containing $B$ are of the form $P_{I}$ for some $I\subseteq S$. Let $W_{I}$ be the subgroup of $W$ generated by $\{s_{\alpha}:\alpha\in I\}$. We note that $W_{I}$ is the Weyl group of $P_{I}$. For $I\subseteq S$, $W^{I}=\{w\in W:w(\alpha)\in R^{+} \text{ for all }\alpha\in I\}$ is the set of minimal length coset representatives of the elements of $W/W_{I}$. Further, there is a natural order on $W^{I}$, namely the restriction of the Bruhat order on $W$. For $w\in W$, we define $R^{+}(w^{-1})=\{\beta\in R^{+}:w^{-1}(\beta)\in R^{-}\}$.\\
Let $I(r,n)=\{(a_{1},a_{2},\ldots,a_{r})\in\mathbb{N}^{r}:1\leq a_{1}<a_{2}<\cdots <a_{r}\leq n\}$. There is a natural order on $I(r,n)$, given by $(a_{1},a_{2},\ldots,a_{r})\leq (b_{1},b_{2},\ldots,b_{r})$ if and only if $a_{i}\leq b_{i}$ for all $1\leq i\leq r$. There is an order preserving identification of $W^{S\setminus\{\alpha_{r}\}}$ with $I(r,n)$, and the correspondence is given by $w\in W^{S\setminus\{\alpha_{r}\}}$ mapping to $(w(1),w(2),\ldots,w(r))$.\\
 For notations and results on standard monomial theory, we refer \cite{Seshadri}.
 
Let $X(T)$ (respectively, $Y(T))$ denote the group of all characters (respectively, one-parametr subgroups) of $T$. Let $E_{1}:=\ X(T)\otimes \mathbb{R}$, and $E_{2}:=\ Y(T)\otimes \mathbb{R}$.\\
 Let $\langle.,.\rangle :\ E_{1}\times E_{2}\longrightarrow\mathbb{R}$ be the canonical non-degenerate form. For every $1\leq j\leq n-1$, there exists $\lambda_{j}\in Y(T)$ such that $\langle\alpha_{i},\lambda_{j}\rangle=\delta_{i,j}$ for all $1\leq i\leq n-1$. Let \[\overline{C(B)}:=\{\lambda\in E_{2}|\ \langle \alpha,\lambda\rangle\geq0\ \text{for all\ }  \alpha\in R^{+}\}.\]	
We have $X(T)\otimes \mathbb{R}=Hom_{\mathbb{R}}(\mathfrak{h}_{\mathbb{R}},\mathbb{R})$, the dual of the real form of $\mathfrak{h}$. The positive definite $W$ - invariant form on $Hom_{\mathbb{R}}(\mathfrak{h}_{\mathbb{R}},\mathbb{R})$ induced by the Killing form on $\mathfrak{g}$ is denoted by $(-,-)$. For any $\mu\in X(T)\otimes\mathbb{R}$ and $\alpha\in R$, denote \[\langle \mu,\alpha\rangle=\frac{2(\mu,\alpha)}{(\alpha,\alpha)}.\]
 Let $\{\omega_{i}:1\leq i\leq n-1\}\subseteq E_{1}$ be the set of fundamental weights, i.e. $\langle\omega_{i},\alpha_{j}\rangle=\delta_{i,j}$ for all $1\leq i,j\leq n-1$.
There is a natural partial order $\le$ on $X(T)$ defined by $\psi\leq \chi$ if and only if $\chi-\psi$ is a nonnegative integral linear combination of simple roots.
Let $u_{\alpha}:\mathbb{C}\longrightarrow U_{\alpha}$ be the isomorphism such that $tu_{\alpha}(a)t^{-1}=u_{\alpha}(\alpha(t)a)$, for all $t\in T$, $a\in \mathbb{C}$. For any $\beta\in R$, $U_{\beta}^{\times}$ denotes the set of non-identity elements of the unipotent group $U_{\beta}$.

A simple root $\alpha_{i}\in S$ is said to be cominuscule if the coefficient of $\alpha_{i}$ in the expression of highest root is $1$. A fundamental weight $\omega_{i}$ is said to be minuscule if $\omega_{i}$ satisfies $\langle\omega_{i},\beta\rangle\leq 1$ for all $\beta\in R^{+}$.\\
Now we recall the definition of semistable point from \cite{GIT}(also see \cite{NEWSTEAD}).

Let $H$ be a reductive algebraic group acting morphically on a projective variety $X$. Let $\mathcal{L}$ be a $H$-linearized very ample line bundle on $X$.
\begin{enumerate}
	\item The set of semistable points  is defined as
	\begin{align*}
		X^{ss}_{H}(\mathcal{L}):=\{x\in X:\exists s\in H^{0}(X,\mathcal{L}^{\otimes m})^{H} \text{ for some $m\in\mathbb{N}$ such that } s(x)\neq 0\}.
	\end{align*}
\item  The set of stable points is defined as
\begin{align*}
	X^{s}_{H}(\mathcal{L}):=\{&x\in 	X^{ss}_{H}(\mathcal{L}): \text{the orbit $H\cdot x$ is closed in 	$X^{ss}_{H}(\mathcal{L})$ and the stabilizer $H_{x}$}\\&  \text{of $x$ in $H$ is finite}\}.
\end{align*}
\end{enumerate}
 Let $x\in \mathbb{P}(H^{0}(X,\mathcal{L})^{*})$ and $\hat{x}$ be a point in the cone $\hat{X}$ over X which lies on $x$. Since $\lambda(\mathbb{G}_{m})$ is a torus, there is a basis  $\{v_{i}: 1\leq i\leq k\}$ of $H^{0}(X,\mathcal{L})^{*}$ and integers $m_{1},m_{2},\ldots,m_{k}$ such that 
$\lambda(t)\cdot v_{i}=t^{m_{i}}v_{i}$, for $1\leq i\leq k$, $t\in \mathbb{G}_{m}$.
Write $\hat{x}=\displaystyle\sum_{i=1}^{k}c_{i}v_{i}$, with $c_{1},\ldots,c_{k}\in\mathbb{C}$. Then the Hilbert-Mumford numerical function is defined by \begin{align*}
	\mu^{\mathcal{L}}(x,\lambda):= -\displaystyle \min_{i}\{m_{i}: c_{i}\neq 0\}.
\end{align*}We recall Hilbert-Mumford criterion.
\begin{theorem}$($see \cite[Theorem 2.1]{GIT}$)$ Let $x\in X$. Then 
\begin{enumerate}
	\item $x\in X^{ss}_{H}(\mathcal{L})$ if and only if $\mu^{\mathcal{L}}(x,\lambda)\geq 0$ for all one parameter subgroup $\lambda$ of $H$.
	\item $x\in X^{s}_{H}(\mathcal{L})$ if and only if $\mu^{\mathcal{L}}(x,\lambda)> 0$ for all non trivial one parameter subgroup $\lambda$ of $H$. 
\end{enumerate}
\label{HM theprem for G}
\end{theorem}
\begin{corollary} Let $H,X, \mathcal{L}$ be as above. Let $\lambda :\mathbb{G}_{m}\rightarrow H$ be an one parameter subgroup. Then, we have 
\begin{enumerate}
	\item $x\in X^{ss}_{\lambda(\mathbb{G}_{m})}(\mathcal{L})$  if and only if both $\mu^{\mathcal{L}}(x,\lambda)$ and $\mu^{\mathcal{L}}(x,-\lambda)$ are non-negative.
	\item $x\in X^{s}_{\lambda(\mathbb{G}_{m})}(\mathcal{L})$ if and only if both $\mu^{\mathcal{L}}(x,\lambda)$ and $\mu^{\mathcal{L}}(x,-\lambda)$ are positive.
\end{enumerate}
\label{HM theorem}
\end{corollary}
Let $G, T, B,$ $\overline{C(B)}$ be as above.
Let $\chi=\displaystyle\sum_{i=1}^{n-1}m_{i}\omega_{i}$ be a non trivial dominant character of $T$. Let $J=\{\alpha_{i}\in S:m_{i}=0\}$. Let $P=P_{J}$, and $w\in W^{J}$, $b\in B$.

For $w\in W^{J}$, let $C(w)=BwP/P$ and $X(w)=\overline{BwP/P}$ denote the Schubert cell and the Schubert variety in $G/P$ corresponding to $w$, respectively.
\begin{lemma}$($See \cite[Lemma 5.1]{Ses2}$)$
Let $x=bwP/P$. 
Let $\lambda\in \overline{C(B)}$ be a one parameter subgroup. Then we have \[\mu^{\mathcal{L}(\chi)}(x,\lambda)=-\langle w(\chi),\lambda\rangle.\]
$($The sign here is negative because we are using left action of P on G/P while in \cite[Lemma 5.1]{Ses2}  the action is on the right.$)$
\label{CSS}
\end{lemma}
	Following variation of the above Lemma follows from   \cite[Lemma 5.1]{Ses2} by imitating the proof for $B^{-}$.
\begin{lemma}	\label{variation of Seshadri's lemma}
	Let $G,T,B,\chi,J, P$ and $\overline{C(B)}$ be as above. Let $w\in W^{J}$, $x\in B^{-}wP/P$. Then for every $\lambda\in \overline{C(B)}$, we have \[\mu^{\mathcal{L}(\chi)}(x,-\lambda)=\langle w(\chi),\lambda\rangle.\]
\end{lemma}
\subsection{Projective bundle}\label{definition:projective bundle} Let  $\pi:\mathcal{V}\longrightarrow X$ be a
vector bundle of rank $m$ on a variety $X$. We recall the definition of projective bundle $\bar{\pi}:\mathbb{P}(\mathcal{V})\longrightarrow X$ on $X$. Let $\mathcal{V}$ has a trivialization $\{(U_{i},\phi_{i}):i\in I\}$ with $\phi_{i}: \pi^{-1}(U_{i})\simeq U_{i}\times \mathbb{C}^{m}$. For each $i,j\in I$, $\phi_{i}\circ\phi_{j}^{-1}$ defines an automorphism of the trivial bundle $(U_{i}\cap U_{j})\times\mathbb{C}^{m}$ and hence a morphism
$g_{ij}:U_{i}\cap U_{j}\longrightarrow GL(m,\mathbb{C})$ such that $\phi_{i}\circ\phi_{j}^{-1}(x,v)=(x,g_{ij}(x)v)$ for all $x\in U_{i}\cap U_{j}$ and $v\in \mathbb{C}^{m}$ . Moreover, the following cocycle conditions hold:\begin{align*}
	g_{ij}(x)\cdot g_{jk}(x)&=g_{ik}(x)\ \text{for all $x\in U_{i}\cap U_{j}\cap U_{k}$}\\
	g_{ij}(x)&=g_{ji}(x)^{-1}\ \text{ for all $x\in U_{i}\cap U_{j}$}
\end{align*}
Note that $g_{ij}$ induces an isomorphism
\begin{align*}
	1\times \bar{g_{ij}}:(U_{i}\cap U_{j})\times \mathbb{P}^{m-1}&\simeq(U_{i}\cap U_{j})\times \mathbb{P}^{m-1}\\
	\text{given by }\bar{g_{ij}}(x,[v])&=(x,[g_{ij}(x)v])
\end{align*} 
Then $1\times \bar{g_{ij}}$ gives gluing data for a variety (see \cite[Ex 2.12, p.80]{Ha}). More precisely, \\
On $\bigsqcup_{i\in I}(U_{i}\times\mathbb{P}^{m-1})$, define an equivalence relation $\sim$ by $(u_{i},[v])\sim (u_{j},[w])$ if and only if $u_{i}=u_{j}$, $([g_{ji}(u_{i})v])=[w]$. Define $\mathbb{P}(\mathcal{V}):=\frac{\bigsqcup_{i\in I}(U_{i}\times\mathbb{P}^{m-1})}{\sim}$ and the equivalence class of $(u_{i},[v])$ is denoted by $[(u_{i},[v])]$.

We have a morphism $\bar{\pi}:\mathbb{P}(\mathcal{V})\longrightarrow X$ given by $\bar{\pi}([(u_{i},[v])])=u$ and $\phi_{i}$ induces trivialization
\begin{align*}
	\bar{\phi_{i}}:\bar{\pi}^{-1}(U_{i})\simeq U_{i}\times\mathbb{P}^{m-1}.
\end{align*}
\subsection{Generalized Bott towers} \label{subsection 2.2}We recall the definition of generalized Bott tower from \cite[A.2, Page 144]{Bott tower}. A generalized Bott tower $X=\{X_{i}:i=0,\ldots,m\}$ with $m$ stages (or of height $m$) is a sequence
\begin{align*}
	X_{m}\xrightarrow{\text{$\pi_m$}}X_{m-1}\xrightarrow{\text{$\pi_{m-1}$}}X_{m-2}\cdots\xrightarrow{\text{$\pi_2$}}X_{1}\xrightarrow{\text{$\pi_{1}$}}X_{0}=\{pt\},
\end{align*}
where for any $1\leq i\leq m$, there exists $n_{i}\in\mathbb{N}$ such that $X_{i}$ is isomorphic to the projective space bundle $\mathbb{P}(\mathcal{E})$ on $X_{i-1}$ and $\mathcal{E}=\mathcal{L}_{i,0}\oplus\mathcal{L}_{i,1}\oplus\mathcal{L}_{i,2}\oplus\cdots\oplus\mathcal{L}_{i,n_{i}}$, for some line bundles $\mathcal{L}_{i,1},\cdots,\mathcal{L}_{i,n_{i}}$ on $X_{i-1}$, and $\mathcal{L}_{i,0}$ is the trivial line bundle on $X_{i-1}$ . We call  $X_{i}$ stage $i$ of the generalized Bott tower $X$.
\section{Combinatorial properties of $w_{r,n}$} In this section, we prove that $X(w_{r,n})$ is the unique minimal dimensional Schubert variety in $G_{r,n}$ admitting semistable points  for the $T_{J_{r}}$-linearized line bundle $\mathcal{L}(n\omega_{r})$.

Fix $r\in\mathbb{N}$. Let $q\in\mathbb{N}$ be such that $q\geq 2$. Let $n=rq+1$. 
 From \cite[Lemma 2.7]{sskannan}, we have 
a reduced expression 
\begin{align*}
	w_{r,n}=(s_{q}s_{q-1}\cdots s_{2}s_{1})(s_{2q}\cdots s_{3}s_{2})\cdots (s_{(r-1)q}\cdots s_{r-1})(s_{rq}\cdots s_{r}).
\end{align*}
Let $T_{J_{r}}$ denote the subtorus of $T$ generated by $\lambda_{jq}(\mathbb{G}_{m})$ ($1\leq j\leq r$).

For every $1\leq j\leq r$, let $$w_{r,n}[j]:=(s_{jq}s_{jq-1}\cdots s_{j})(s_{jq+1}s_{jq}\cdots s_{j+1})\cdots(s_{jq+r-j}s_{jq+r-j-1}\cdots s_{r}).$$ Then, we have the following.
\begin{lemma}\label{section:combinatorial result:1}For every $1\leq j\leq r$, $X(w_{r,n}[j])$ is the unique minimal dimensional Schubert variety in $G_{r,n}$ admitting semistable points for the $\lambda_{jq}(\mathbb{G}_{m})$-linearized line bundle $\mathcal{L}(n\omega_{r})$.
\end{lemma}
\begin{proof}
	Fix $j\in\{1,2,\ldots,r\}$. Since $n=rq+1$, we have \begin{align}\label{eq 1.1.6}
		\frac{(jq)r}{n}=\frac{j(n-1)}{n}<j
	\end{align}
\begin{align}\label{eq 1.1.7}
	\frac{(jq)r}{n}-(j-1)=\frac{jqr-jn+n}{n}=\frac{n-j}{n}>0
\end{align}
From \cref{eq 1.1.6} and \cref{eq 1.1.7}, we have 
\begin{align*}
	j-1<\frac{(jq)r}{n}<j.
\end{align*}
Therefore, we have $\lfloor\frac{(jq)r}{n}\rfloor=j-1$. Now the proof follows from \cite[Lemma 5.8]{GKgit}.
\end{proof}
\begin{lemma}\label{sction:combinatorial resul:2}
	Assume that $(r,n)=1$. Let $w_{r,n}$ be the unique minimal element in $W^{S\setminus\{\alpha_{r}\}}$ such that $X(w_{r,n})^{ss}_{T}(\mathcal{L}(n\omega_{r}))\neq \phi$. Let $\mathrm{Peaks}(w_{r,n})=\{\alpha_{i_{j}}:1\leq j\leq l\}$ and $T_{J}$ denote the subtorus of $T$ generated by $\lambda_{i_{j}}(\mathbb{G}_{m})$ $($$1\le j\le l$$)$. Then $X(w_{r,n})$ is the unique minimal element of $W^{S\setminus\{\alpha_{r}\}}$ such that $X(w_{r,n})^{ss}_{T_{J}}(\mathcal{L}(n\omega_{r}))\neq \phi$.
\end{lemma}
\begin{proof}
	Since $X(w_{r,n})^{ss}_{T}(\mathcal{L}(n\omega_{r}))\neq \phi$, there exists a natural number $d$ and a non zero section $s\in H^{0}(X(w_{r,n}),\mathcal{L}(dn\omega_{r}))^{T}$.\\ Since $T_{J_{r}}\subseteq T$, we have $s\in 	H^{0}(X(w_{r,n}),\mathcal{L}(dn\omega_{r}))^{T_{J_{r}}}$. Hence, $X(w_{r,n})^{ss}_{T_{J_{r}}}(\mathcal{L}(n\omega_{r}))\neq \phi$.
	Since $X(w_{r,n})^{ss}_{T}(\mathcal{L}(n\omega_{r}))\neq \phi$ and $Bw_{r,n}P/P$ is open, there exists a point $x\in X(w_{r,n})^{ss}_{T}(\mathcal{L}(n\omega_{r}))\cap Bw_{r,n}P/P$.\\ Since $X(w_{r,n})^{ss}_{T}(\mathcal{L}(n\omega_{r}))=X(w_{r,n})^{s}_{T}(\mathcal{L}(n\omega_{r}))$, from \cref{HM theprem for G}, we have  $\mu^{\mathcal{L}(n\omega_{r})}(x,\lambda)>0$ for all $\lambda\in Y(T)\setminus\{0\}$. Therefore, $\mu^{\mathcal{L}(n\omega_{r})}(x,\lambda)>0$ for all $\lambda\in Y(T_{J})$. In particular, $\mu^{\mathcal{L}(n\omega_{r})}(x,\lambda_{i_{j}})>0$ for all $1\leq j\leq l$. By \cref{CSS}, we have $\mu^{\mathcal{L}(n\omega_{r})}(x,\lambda_{i_{j}})=-\langle w_{r,n}(n\omega_{r}),\lambda_{i_{j}}\rangle$. Hence, $\langle w_{r,n}(n\omega_{r}),\lambda_{i_{j}}\rangle<0$. Since $w_{r,n}$ is a minimal element in $W^{S\setminus\{\alpha_{r}\}}$ such that $X(w_{r,n})^{ss}_{T}(\mathcal{L}(n\omega_{r}))\neq \phi$, for every $1\leq j\leq l$, $\langle s_{i_{j}}w_{r,n}(\omega_{r}),\lambda_{i_{j}}\rangle>0$. Let $v$ be a minimal element of $W^{S\setminus\{\alpha_{r}\}}$ such that $X(v)^{ss}_{T_{J}}(\mathcal{L}(n\omega_{r})\neq\phi$. Then the number of times appearences of $s_{i_{j}}$'s $($ $1\leq j\leq l$ $)$ in $v$ is equal to the number of times appearing in $w_{r,n}$. Therefore, $v=w_{r,n}$.
\end{proof}
We illustrate \cref{semistability criterion} and \cref{main result 1} in the following example.
\begin{example}Let $r=q=3$.  Let $\mathcal{L}$ denote the restriction of the ample line bundle $\mathcal{O}(10)$ to $G_{3,10}$ associated to the Plücker embedding
	$G_{3,10} \hookrightarrow \mathbb{P}\!\left( \wedge^{3} \mathbb{C}^{10} \right).$
	
Let $\mathcal{L}^{\prime}$ denote the restriction of the ample line bundle $\mathcal{O}(7)$ to $G_{2,7}$ associated to the Plücker embedding
$G_{2,7} \hookrightarrow \mathbb{P}\!\left( \wedge^{2} \mathbb{C}^{7} \right).$

In the first step, we identify semistable locus.
In one line notation $w_{3,10}=(4,7,10)$ and $w_{2,7}=(4,7)$.
The Schubert cell $C(w_{3,10})$ can be identified with the subset of $M_{10\times 3}(\mathbb{C})$ consisting of matrices of the form
\begin{align*}
	A=\begin{pmatrix}
		a_{1,1}&a_{1,2}&a_{1,3}\\
		a_{2,1}&a_{2,2}&a_{2,3}\\
		a_{3,1}&a_{3,2}&a_{3,3}\\
		1&0&0\\
		0&a_{5,2}&a_{5,3}\\
		0&a_{6,2}&a_{6,3}\\
		0&1&0\\
		0&0&a_{8,3}\\
		0&0&a_{9,3}\\
		0&0&1
	\end{pmatrix}
\end{align*}
 In view of \cref{sction:combinatorial resul:2}, semistable locus is contained in the open Schubert cell $C(w_{3,10})$. 
Let $A=(a_{i,j})\in C(w_{3,10})$. Then $A\in X(w_{3,10})^{ss}_{T_{J_{3}}}(\mathcal{L})$ if and only if \begin{enumerate}
	\item at least one of $\{a_{1,1}, a_{2,1}, a_{3,1}\}$ is non zero, and
	\item at least one of $\{a_{1,2},a_{2,2},a_{3,2},a_{5,2},a_{6,2}\}$ is non zero, and
	\item at least one of $\{a_{1,3},a_{2,3},a_{3,3},a_{5,3},a_{6,3},a_{8,3},a_{9,3}\}$ is non zero.
\end{enumerate}

The Schubert cell $C(w_{2,7})$ can be identified with the subset of $M_{7\times 2}(\mathbb{C})$ consiting of matrices of the form
\begin{align*}
	B=\begin{pmatrix}
		b_{1,1}&b_{1,2}\\
		b_{2,1}&b_{2,2}\\
		b_{3,1}&b_{3,2}\\
		1&0\\
		0&b_{5,2}\\
		0&b_{6,2}\\
		0&1
	\end{pmatrix}
\end{align*}
Similarly, for $B=(b_{i,j})\in C(w_{2,7})$, $B\in X(w_{2,7})^{ss}_{T_{J_{2}}}(\mathcal{L}^{\prime})$ if and only if 
\begin{enumerate}
	\item at least one of $\{b_{1,1},b_{2,1},b_{3,1}\}$ is non zero, and
	\item at least one of $\{b_{1,2},b_{2,2},b_{3,2},b_{5,2},b_{6,2}\}$ is non zero.
\end{enumerate}
In the second step, we show that the semistable locus coincides with the stable locus.\\ Thus, the GIT quotients  $\GmodX{T_{J_{3}}}{(X(w_{3,10}))}{\mathcal{L}}$ and $\GmodX{T_{J_{2}}}{(X(w_{2,7}))}{\mathcal{L}^{\prime}}$ are geometric quotients, and are respectively isomorphic to, $T_{J_{3}}\backslash C(w_{3,10})^{s}$and  $T_{J_{2}}\backslash C(w_{2,7})^{s}$.

We define a morphism $\phi:C(w_{3,10})\longrightarrow C(w_{2,7})$ by
\begin{align*}
	\phi(\begin{pmatrix}
		a_{1,1}&a_{1,2}&a_{1,3}\\
		a_{2,1}&a_{2,2}&a_{2,3}\\
		a_{3,1}&a_{3,2}&a_{3,3}\\
		1&0&0\\
		0&a_{5,2}&a_{5,3}\\
		0&a_{6,2}&a_{6,3}\\
		0&1&0\\
		0&0&a_{8,3}\\
		0&0&a_{9,3}\\
		0&0&1
	\end{pmatrix})=\begin{pmatrix}
	a_{1,1}&a_{1,2}\\
	a_{2,1}&a_{2,2}\\
	a_{3,1}&a_{3,2}\\
	1&0\\
	0&a_{5,2}\\
	0&a_{6,2}\\
	0&1
\end{pmatrix}
\end{align*}
Then $\phi$ descends to a morphism $\tilde{\phi}:T_{J_{3}}\backslash C(w_{3,10})^{s}\longrightarrow T_{J_{2}}\backslash C(w_{2,7})^{s}$ and  $\tilde{\phi}$ is a locally trivial $\mathbb{P}^{6}$-bundle.
\end{example}
	 \section{Description of  $X(w_{r,n})^{ss}_{T_{J_{r}}}(\mathcal{L}(n\omega_{r}))$ for $n\equiv1 (mod \ r)$} In this section, we find a description of the semistable locus $X(w_{r,n})^{ss}_{T_{J_{r}}}(\mathcal{L}(n\omega_{r}))$. Further, we prove that $X(w_{r,n})^{ss}_{T_{J_{r}}}(\mathcal{L}(n\omega_{r}))=X(w_{r,n})^{s}_{T_{J_{r}}}(\mathcal{L}(n\omega_{r}))$ and  $\GmodX{T_{J_{r}}}{X(w_{r,n})}{\mathcal{L}(n\omega_{r})}$ is smooth.\\
 For the simplicity of notation, throughout this section, we denote $w_{r,n}$ by $w$ and $P_{S\setminus\{\alpha_{r}\}}$ by $P$.\\  Recall that a reduced expression of $w_{r,n}$ is given by
 \begin{align*}
 	w_{r,n}=(s_{q}s_{q-1}\cdots s_{2}s_{1})(s_{2q}\cdots s_{3}s_{2})\cdots (s_{(r-1)q}\cdots s_{r-1})(s_{rq}\cdots s_{r}).
 \end{align*}
Fix an integer $j$ such that $1\leq j\leq r$. Let
\begin{align*}
	\alpha_{jq,j}:&=\alpha_{jq}\\
	\text{for $j\leq i\leq jq-1$, }\alpha_{i,j}:&=(s_{q}\cdots s_{1})\cdots(s_{(j-1)q}\cdots s_{j-1})(s_{jq}\cdots s_{i+1}(\alpha_{i}))
\end{align*} 
Then from \cite[Lemma 8.3.2(i)]{springer}, we have $R^{+}(w^{-1})=\{\alpha_{i,j}:1\leq j\leq r;j\leq i\leq jq\}$.\\ Note that the elements of $R^{+}(w^{-1})$ are of the form 
$\alpha_{l}+\alpha_{l+1}+\cdots+\alpha_{jq}$, where
$1\leq j\leq r$, and $l\neq kq+1$ for any $1\leq k\leq j-1$ (see \cite[Lemma 2.4]{sskannan}).\\
For $1\leq j\leq r$, define $C_{j}:=\{1,2,\ldots,jq+1\}\setminus\{iq+1: 1\leq i\leq j\}$. For any  $i\in C_{j}$, define $\beta_{i,j}=\sum_{k=i}^{jq}\alpha_{k}$. Then, we have $R^{+}(w^{-1})=\{\beta_{i,j}:1\leq j\leq r, i\in C_{j}\}$. \\
Further, for every fixed $j\in\{1,2,\ldots,r\}$, we have 
\begin{align}\label{Special eq 2}
	\{\alpha_{i,j}:j\leq i\leq jq\}=\{\beta_{i,j}:i\in C_{j}\}
\end{align}
 Let $U_{w}=\prod_{\beta_{i,j}\in R^{+}(w^{-1}) }U_{\beta_{i,j}}$, where the order in which the product is taken is as follows: define a total order on $R^{+}(w^{-1})$ by $\beta_{i,j}\leq \beta_{k,l}$ if either $j=l$ and $i\leq k$ or if $j<l$. Now we take the product so that whenever $\beta_{i,j}\leq \beta_{k,l}$, $u_{\beta_{i,j}}(a_{i,j})$ appears on the left-hand side to $u_{\beta_{k,l}}(a_{k,l})$, where $a_{i,j},\ a_{k,l}\in\mathbb{C}$.
 \begin{lemma}\label{remark 1.1}
 	$u_{\beta_{i,j}}(a_{i,j})$ commutes with $u_{\beta_{k,l}}(a_{k,l})$ for all $\beta_{i,j},\beta_{k,l}\in R^{+}(w^{-1})$, $a_{i,j},\ a_{k,l}\in\mathbb{C}$.
 \end{lemma}
\begin{proof} It suffices to show that, for $\beta_{i,j},\beta_{k,l}\in R^{+}(w^{-1})$, $\beta_{i,j}+\beta_{k,l}\notin R^{+}(w^{-1})$. For a proof we refer to \cite{sskannan} (see discussion before Lemma 2.6).
\end{proof}
Let $X_{\beta_{i,j}}$ be the coordinate function on $U_{w}$ corresponding to $\beta_{i,j}\in R^{+}(w^{-1})$. That is $X_{\beta_{i,j}}(\prod_{\substack{1\leq l\leq r\\k\in C_{l}}}u_{\beta_{k,l}}(a_{k,l}))=a_{i,j}$, where $a_{k,l}\in\mathbb{C}$ for all $1\leq l\leq r$ and $k\in C_{l}$. Then the coordinate ring $\mathbb{C}[U_{w}]$ of $U_{w}$ is the polynomial ring $\mathbb{C}[X_{\beta_{i,j}}:1\leq j\leq r, i\in C_{j}]$.

Note that, in one line notation, $w= (q+1,2q+1,\ldots,rq+1)$. 
The Schubert cell $C(w)$ can be identified with the following set of matrices of $M_{n\times r}(\mathbb{C})$ (see \cite[section 1.2.1]{Seshadri}).
\begin{equation}\label{eq 1.0}
	C(w)=
	\left\{
	(a_{i,j})\in M_{n\times r}(\mathbb{C})\Bigg\vert
	\begin{matrix}
		(i) a_{jq+1,k}=\delta_{j,k} &\text{ for $1\leq j\leq r$
			and $1\leq k\leq r$}\\ 
		(ii) a_{i,j}=0 &\text{ for $1\leq j\leq r-1$ and $jq+2\leq i\leq rq+1$} \\
	\end{matrix}
	\right\}
\end{equation}

The map $\psi_{1}:C(w)\longrightarrow U_{w}$ defined by $$\psi_{1}(A)=\displaystyle\prod_{\substack{1\leq j\leq r\\i\in C_{j}}}u_{\beta_{i,j}}(a_{i,j}),$$ for all $A=(a_{i,j})_{\substack{1\leq i\leq n\\ 1\leq j\leq r}}\in C(w)$, 
is an isomorphism of varieties. \\
Let $\psi_{2}:U_{w}\longrightarrow X(w)$ 
be defined by $\psi_{2}(u)=uwP/P$ for all $u\in U_{w}$. Then $\psi_{2}$ is an isomorphism onto $U_{w}wP/P$. 

Consider the natural action of $T$ on $\mathbb{C}[U_{w}]$. Now,
identify $\mathbb{C}[U_{w}]$ with $\mathbb{C}[U_{w}]\otimes_{\mathbb{C}}\mathbb{C}$. Consider the action of $\hat{T}$ on $\mathbb{C}[U_{w}]\otimes_{\mathbb{C}}\mathbb{C}$ through the homomorphism $\pi:\hat{T}\longrightarrow T$ is given by \begin{align}\label{eq 1.1.1}
	t\cdot (f\otimes c)=(\pi(t)\cdot f)\otimes-w(\omega_{r})(t)c
\end{align}
for all $t\in \hat{T}$, $f\in\mathbb{C}[U_{w}]$ and $c\in\mathbb{C}$.
For the natural action of $\hat{T}$ on $H^{0}(X(w),\mathcal{L}(\omega_{r}))$ and the twisted action of $\hat{T}$ on $\mathbb{C}[U_{w}]$ as in \cref{eq 1.1.1}, $\psi_{2}^{*}:H^{0}(X(w),\mathcal{L}(\omega_{r}))\longrightarrow\mathbb{C}[U_{w}]$ is a homomorphism of $\hat{T}$-modules.
\begin{lemma}\label{remark 1.2}Let $v\leq w$ in $W^{S\setminus\{\alpha_{r}\}}$. For the natural action of $\hat{T}$ on $\mathbb{C}[U_{w}]$, the weight of $\psi_{2}^{*}(p_{v})$ is $-v(\omega_{r})+w({\omega_{r}})$.
\end{lemma}
	\begin{proof}
	  Proof follows from the identification of $\mathbb{C}[U_{w}]$ with $\mathbb{C}[U_{w}]\otimes\mathbb{C}$ and \cref{eq 1.1.1}.
	\end{proof}

 	 \begin{lemma} \label{Lemma 4.1}
 	For $1\leq i\leq r$,
 	we have $\langle w(n\omega_{r}),\lambda_{iq}\rangle=-(n-i)$ .
 \end{lemma}
 \begin{proof}
 	Let $w(\omega_{r})=\omega_{r}-\sum_{k=1}^{n-1}a_{i}\alpha_{i}$, where $a_{i}\in\mathbb{Z}_{\geq 0}$. Note that $a_{i}$ is the number of times the simple reflection $s_{i}$ appears in a reduced expression of $w$ (see \cite[Corollary 1.4]{sskannan}). \\
 	Then we have $\langle w(n\omega_{r}),\lambda_{iq}\rangle=\langle n\omega_{r},\lambda_{iq}\rangle-na_{iq}$.
 	From the description of $\omega_{r}$ in \cite[p.69]{Hum1}, we have
 	\begin{align*}
 		\langle n\omega_{r},\lambda_{iq}\rangle&=\begin{dcases}
 		r(n-iq)&\text{\ if $r\leq iq$}\\
 		iq(n-r)&\text{\ if $iq<r$}
 		\end{dcases}
 	\end{align*}
 	{\it Case 1}: Let $r\leq iq$. In this case $a_{iq}=r-i+1$. Therefore, we have
 	\begin{align*}
 		\langle w(n\omega_{r}),\lambda_{iq}\rangle=r(n-iq)-n(r-i+1)
 		&=-irq+in-n\\
 		&=i-n\text{\ (using $n=rq+1$)}\\
 		&=-(n-i).
 	\end{align*}
 	{\it Case 2}: Let $r>iq$. In this case $a_{iq}=iq-i+1$. Therefore, we have
 	\begin{align*}
 		\langle w(n\omega_{r}),\lambda_{iq}\rangle=iq(n-r)-n(iq-i+1)
 		&=-irq+in-n\\
 		&=i-n(\ \text{using $n=rq+1$})\\
 		&=-(n-i).
 	\end{align*}
 \end{proof}
\subsection{Partition of the set $C_{j}\times\{j\}$}
For every $1\leq j\leq r$ and for every $1\leq p\leq j$, we define the sets $J_{p,j}$ in the following way:
\begin{align*}
	\text{For $1\leq j\leq r$, define }  J_{1,j}:&=\{(i,j):1\leq i\leq q\}.\\
	\text{For $2\leq j\leq r$ and for $2\leq p\leq j$, define } J_{p,j}:&=\{(i,j):(p-1)q+2\leq i\leq pq\}.
\end{align*}
\begin{lemma}\label{remark 1}
	For any $1\leq j\leq r$, we have $C_{j}\times\{j\}=\bigsqcup_{p=1}^{j}J_{p,j}$.
\end{lemma}
\begin{proof}
	Let $(i,j)\in J_{p,j}$. Then, $i\leq pq$. Since $p\leq j$, we have $i\leq jq$. Further,  $i\neq p^{\prime}q+1$ for every $1\leq p^{\prime}\leq j$. Hence, we have $(i,j)\in C_{j}\times \{j\}$.\\
	Conversely, let $(i,j)\in C_{j}\times\{j\}$. This implies $i\neq kq+1$ for all $1\leq k\leq j$. Let $p$ be the least positive integer such that $i\leq pq$. Therefore, $i>(p-1)q$. Since $i\neq (p-1)q+1$, we have $(i,j)\in J_{p,j}$.
\end{proof}
\subsection{Construction of $e^{k}(i_{j},j)$}\label{construction of e^{k}}
For a fixed $1\leq j\leq r$, fix an  $i_{j}\in C_{j}$ arbitrarily. From \cref{remark 1}, we have, for every $1\leq j\leq r$, there exists a unique integer $ d(i_{j},j)\in\{1,2,\ldots,j\}$ such that $(i_{j},j)\in J_{d(i_{j},j),j}$. 
\begin{align*}
	\text{Define } & e^{0}(i_{j},j):=j.
\end{align*}
By definition of $e^{0}(i_{j},j)$, there is a non negative integer $k$ such that $e^{k}(i_{j},j)\in\mathbb{Z}_{\geq0}$. So, let $m(i_{j},j)$ be the largest non negative integer for which $e^{m(i_{j},j)}(i_{j},j)$ is defined and non negative integer. Now, we iteratively define $e^{k}(i_{j},j)$ for every $0\leq k\leq m(i_{j},j)$, 
\begin{align}\label{Special eq 6}
	\text{}& e^{k}(i_{j},j):= d(i_{_{e^{k-1}(i_{j},j)}},e^{k-1}(i_{j},j))-1.	
\end{align}
Thus, for every $0\leq k\leq m(i_{j},j)$, we have a map $e^{k}:C_{j}\times\{j\}\longrightarrow[0,j]\cap\mathbb{Z}_{\geq 0}$. In the following Lemma , we prove that for every $(i_{j},j)$, $e^{k}(i_{j},j)$ is a decreasing function on $k$.
 \begin{lemma}\label{lemma 2.1}
 	 We have $e^{k}(i_{j},j)<e^{k-1}(i_{j},j)$.
 \end{lemma}
\begin{proof}
	 From the defintion of $e^{k}(i_{j},j)$, we have $e^{k}(i_{j},j)= d(i_{_{e^{k-1}(i_{j},j)}},e^{k-1}(i_{j},j))-1$. Since $d(i_{_{e^{k-1}(i_{j},j)}},e^{k-1}(i_{j},j))\leq e^{k-1}(i_{j},j)$, we get $e^{k}(i_{j},j)\leq e^{k-1}(i_{j},j)-1$.
\end{proof}
\begin{lemma}\label{e^{m}=0}
	For every $(i_{j},j)$, we have $e^{m(i_{j},j)}(i_{j},j)=0$.
\end{lemma}
\begin{proof}
	Recall that $m(i_{j},j)$ be the largest non negative integer such that $e^{m(i_{j},j)}(i_{j},j)\geq 0$. We prove that $e^{m(i_{j},j)}(i_{j},j)=0$. Assume on the contrary that $e^{m(i_{j},j)}(i_{j},j)\geq 1$. Then, we have \begin{align}\label{eq 1.1}
	e^{m(i_{j},j)+1}(i_{j},j)=d(i_{e^{m(i_{j},j)}(i_{j},j)},e^{m(i_{j},j)}(i_{j},j))-1.
\end{align}
Since $e^{m(i_{j},j)}(i_{j},j)\geq 1$, we have $d(i_{e^{m(i_{j},j)}(i_{j},j)},e^{m(i_{j},j)}(i_{j},j))\geq 1$. Thus from \cref{eq 1.1}, we have $e^{m(i_{j},j)+1}(i_{j},j)\geq 1-1=0$. This contradicts the maximality of $m(i_{j},j)$. Hence, $e^{m(i_{j},j)}(i_{j},j)=0$.
\end{proof}
Since $e^{0}(i_{j},j)=j\geq 1$ and $e^{m(i_{j},j)}(i_{j},j)=0$, we have $m(i_{j},j)>0$.
\begin{lemma}\label{claim 1}
	Fix $1\leq j\leq r$. Then for every $1\leq l\leq r$, we have 
	\begin{align*}
		\langle \beta_{i_{j},j},\lambda_{lq}\rangle&=\begin{dcases}
			1&\text{ if }d(i_{j},j)\leq l\leq j\\
			0&\text{ otherwise} 
		\end{dcases}
	\end{align*}
\end{lemma}
\begin{proof}
	We have $\beta_{i_{j},j}=\sum_{k=i_{j}}^{jq}\alpha_{k}$.
	Since $(i_{j},j)\in J_{d(i_{j},j),j}$,  we have
	\begin{align}\label{eq 1.13}
		(d(i_{j},j)-1)q+1\leq i_{j}\leq d(i_{j},j)q
	\end{align}
	Further, from \cref{eq 1.13}, and using $d(i_{j},j)\leq j$, we have $d(i_{j},j)$ is the smallest positive integer $m$ such that $\alpha_{mq}\leq \beta_{i_{j},j}$. From the definition of $\beta_{i_{j},j}$, we have $\alpha_{lq}\nleq \beta_{i_{j},j}$ for all $j+1\leq l\leq r$. Therefore, $\alpha_{lq}\leq \beta_{i_{j},j}$ if and only if $d(i_{j},j)\leq l\leq j$. Now the proof follows from the fact that $\alpha_{lq}$ is a cominuscule simple root.
\end{proof}
For every $1\leq j\leq r$, define \begin{align}\label{eq 1.14}
	\gamma_{j}:=\sum_{k=0}^{m(i_{j},j)-1}\beta_{i_{e^{k}(i_{j},j)},e^{k}(i_{j},j)}
\end{align} 
\begin{lemma}\label{claim 4}
	Let $1\leq j\leq r$. 
	For any $1\leq l\leq r$, we have \begin{align*}
		\langle \gamma_{j},\lambda_{lq}\rangle &=	\begin{dcases}
			1& \text{ if } 1\leq l\leq j\\
			0& \text{ if }j+1\leq  l\leq r
		\end{dcases}
	\end{align*}
\end{lemma}
\begin{proof} Step 1: We prove that $\langle\gamma_{j},\lambda_{lq}\rangle=0$ for every integer $l$ such that $j+1\leq l\leq r$.
	
	Let $j+1\leq l\leq r$.
	Since $e^{0}(i_{j},j)=j$ and $e^{k}(i_{j},j)<e^{0}(i_{j},j)$ for all $1\leq k\leq m(i_{j},j)-1$, we have $l>e^{k}(i_{j},j)$ for $0\leq k\leq m(i_{j},j)-1$. Therefore, from \cref{claim 1}, we have  $\langle\beta_{i_{e^{k}(i_{j},j)},e^{k}(i_{j},j)},\lambda_{lq}\rangle=0$ for $0\leq k\leq m(i_{j},j)-1$. Hence, from \cref{eq 1.14}, we have $\langle\gamma_{j},\lambda_{lq}\rangle=0$.\\
	Step 2: We prove that $\langle\gamma_{j},\lambda_{lq}\rangle=1$, for every integer $l$ such that $1\leq l\leq j$.
	
	Let $1\leq l\leq j$. Since $e^{0}(i_{j},j)=j$, using \cref{lemma 2.1} and \cref{e^{m}=0}, we have 
	\begin{align*}
		\{1,2,\ldots,j\}=\sqcup_{k=0}^{m(i_{j},j)-1}\{m\in \mathbb{N}:e^{k+1}(i_{j},j)+1\leq m\leq e^{k}(i_{j},j)\}
	\end{align*}
	Let $k_{0}\in\{0,\ldots,m(i_{j},j)-1\}$ be the unique non negative integer such that \begin{align}\label{eq 1.12}
		e^{k_{0}+1}(i_{j},j)+1\leq l\leq e^{k_{0}}(i_{j},j)
	\end{align}
	Since 	$e^{k_{0}+1}(i_{j},j)+1=d(i_{e^{k_{0}}(i_{j},j)},e^{k_{0}}(i_{j},j))$, from \cref{claim 1}, we have $$\langle\beta_{i_{e^{k_{0}(i_{j},j)}},e^{k_{0}(i_{j},j)}},\lambda_{lq}\rangle=1.$$ Now we prove that $\langle\beta_{i_{e^{k}(i_{j},j)},e^{k}(i_{j},j)},\lambda_{lq}\rangle=0$ for $k\in\{0,\ldots,m(i_{j},j)-1\}\setminus\{k_{0}\}$.
	First assume that $k<k_{0}$. Then by \cref{lemma 2.1}, we have $$e^{k_{0}}(i_{j},j)<e^{k+1}(i_{j},j)+1=d(i_{e^{k}(i_{j},j),e^{k}(i_{j},j)}).$$ Therefore, from \cref{eq 1.12}, we have $l<d(i_{e^{k}(i_{j},j),e^{k}(i_{j},j)})$. Hence, using \cref{claim 1}, we have $\langle\beta_{i_{e^{k}(i_{j},j)},e^{k}(i_{j},j)},\lambda_{lq}\rangle=0$.\\
	Now assume that $k>k_{0}$. Then from \cref{lemma 2.1}, we have $e^{k}(i_{j},j)\leq e^{k_{0}+1}(i_{j},j)$. Therefore from \cref{eq 1.12}, we have $l>e^{k}(i_{j},j)$. Hence from \cref{claim 1}, we have $\langle\beta_{i_{e^{k}(i_{j},j)},e^{k}(i_{j},j)},\lambda_{lq}\rangle=0$. Thus, for $1\leq l\leq j$, we have $$\langle\gamma_{j},\lambda_{lq}\rangle=\langle \beta_{i_{e^{k_{0}}(i_{j},j)},e^{k_{0}}(i_{j},j)},\lambda_{lq}\rangle=1.$$
\end{proof}
\subsection{Construction of a $T_{J_{r}}$-invariant section}
Let $x=uwP/P$ with $u\in U_{w}$. Assume that for every $1\leq j\leq r$, there is an integer $i_{j}\in C_{j}$ such that $X_{\beta_{i_{j},j}}(u)\neq 0$. In this section, we prove that there exists $s\in H^{0}(X(w),\mathcal{L}(n\omega_{r}))^{T_{J_{r}}}$ such that $s(x)\neq 0$.
	\begin{lemma}\label{Special lemma 1}
		 For every $1\leq j\leq r$, fix such an integer $i_{j}\in C_{j}$. Let $\gamma_{j}=\beta_{i_{j},j}$. Then there exists a monomial $M_{r}$ in $X_{\gamma_{j}}$'s $(1\leq j\leq r)$ satisfying
		\begin{enumerate}
			\item $M_{r}\in \psi_{2}^{*}(H^{0}(X(w),\mathcal{L}((n-r)\omega_{r}))$
			\item the weight $\mu_r$ of the monomial $M_{r}$ satisfies $\langle\mu_r,\lambda_{jq}\rangle=-(n-r)$ for every $1\leq j\leq r$.
				\end{enumerate}
\end{lemma}
\begin{proof}
	Let $k_{1}$ be the least positive integer such that $i_{r}\leq k_{1}q$. If $k_{1}=1$, then $\alpha_{q}\leq \beta_{i_{r},r}$. Therefore, we take $M_{r}=X_{\gamma_{r}}^{n-r}$. Then, since $X_{\gamma_{r}}\in \psi_{2}^{*}(H^{0}(X(w),\mathcal{L}(\omega_{r})))$, we have $M_{r}\in \psi_{2}^{*}(H^{0}(X(w),\mathcal{L}((n-r)\omega_{r}))$. Furter, $\langle\mu_{r},\lambda_{jq}\rangle=-(n-r)$ for every $1\leq j\leq r$. Therefore, $M_{r}$ satisfies conditions $(1)$ and $(2)$.\\
	If $k_{1}\geq 2$, let $k_{2}$ be the least positive integer such that $i_{k_{1}-1}\leq k_{2}q$. Since $i_{k_{1}-1}\leq (k_{1}-1)q$, we have $k_{2}<k_{1}$.\\ If $k_{2}=1$, then $\alpha_{q}\leq \gamma_{k_{1}-1}$. Therefore, we take $M_{r}=(X_{\gamma_{r}}X_{\gamma_{k_{1}-1}})^{n-r}$. 
	To see that $X_{\gamma_{r}}X_{\gamma_{k_{1}-1}}\in \psi_{2}^{*}(H^{0}(X(w),\mathcal{L}(\omega_{r}))$, we note that, there are integers $a_{1},a_{2},\ldots,a_{r-2}\in \{1,2,\ldots,n\}\setminus\{i_{r},i_{k_{1}-1}\}$ and integers $b_{1},b_{2},\ldots,b_{r-2}\in\{1,2,\ldots,r\}\setminus\{r,k_{1}-1\}$ such that for every $i\in \{1,2,\ldots,r-2\}$, there exists a unique $b_{l(i)}$, $l(i)\in\{1,2,\ldots,r-2\}\setminus\{r,k_{1}-1\}$ such that $(a_{i},b_{j})^{th}$ entry is $\delta_{l(i),j}$. Then we have $\psi_{2}^{*}(\pm p_{\{a_{1},a_{2},\ldots,a_{r-2},i_{r},i_{k_{1}-1}\}\uparrow})=X_{\gamma_{r}}X_{\gamma_{k_{1}-1}}$.
	Therefore, $M_{r}$ satisfies conditions $(1)$ and $(2)$. \\
	If $k_{2}\geq 2$, let $k_{3}$ be the least positive integer such that $i_{k_{2}-1}\leq k_{3}q$. Since $i_{k_{2}-1}\leq (k_{2}-1)q$, we have $k_{3}<k_{2}$.\\
	If $k_{3}=1$, we take $M_{r}=(X_{\gamma_{r}}X_{\gamma_{k_{1}-1}}X_{\gamma_{k_{2}-1}})^{n-r}$, otherwise proceeding reursively, there is a decreasing sequence 
	$$r\geq k_{1}>k_{2}>k_{3}>\cdots>k_{l}=1$$
	such that $i_{k_{j}-1}\leq k_{j+1}q$ for every $1\leq j\leq l-1$. \\
	Now take $M_{r}=(X_{\gamma_{r}}X_{\gamma_{k_{1}-1}}X_{\gamma_{k_{2}-1}}\cdots X_{\gamma_{k_{l-1}-1}})^{n-r}$. 
	By a similar argument as in the proof of $X_{\gamma_{r}}X_{\gamma_{k_{1}-1}}\in \psi_{2}^{*}(H^{0}(X(w),\mathcal{L}(\omega_{r}))$, we can see that  $X_{\gamma_{r}}X_{\gamma_{k_{1}-1}}X_{\gamma_{k_{2}-1}}\cdots X_{\gamma_{k_{l-1}-1}}\in \psi_{2}^{*}(H^{0}(X(w),\mathcal{L}(\omega_{r})))$.
	Therefore $M_{r}\in \psi_{2}^{*}(H^{0}(X(w),\mathcal{L}((n-r)\omega_{r}))$ and the weight $\mu_{r}$ of $M_{r}$ satisfies $\langle\mu_{r},\lambda_{jq}\rangle=-(n-r)$ for every $1\leq j\leq r$. Therefore, $M_{r}$ satisfies conditions $(1)$ and $(2)$.
\end{proof}
\begin{lemma}\label{Special lemma 2} There is a monomial $M^{\prime}_{r-1}$ in $X_{\gamma_{j}}$'s $(1\leq j\leq r-1)$ such that 
	\begin{enumerate}
		\item $M^{\prime}_{r-1}\in \psi_{2}^{*}(H^{0}(X(w),\mathcal{L}(\omega_{r})))$ and
		\item the weight $\mu^{\prime}_{r-1}$ of $M^{\prime}_{r-1}$ satisfies $\langle\mu^{\prime}_{r-1},\lambda_{jq}\rangle=-1$ for every $1\leq j\leq r-1$ and $\langle\mu^{\prime}_{r-1},\lambda_{rq}\rangle=0$
	\end{enumerate}
\end{lemma}
\begin{proof}
Proof is similar to that of the construction of $M_{r}$ in \cref{Special lemma 1}. Let $k_{1}$ be the least positive integer such that $i_{r-1}\leq k_{1}q$. If $k_{1}=1$, we take $M^{\prime}_{r-1}=X_{\gamma_{r-1}}$. Clearly, $M^{\prime}_{r-1}$ satisfies conditions $(1)$ and $(2)$. Otherwise, let $k_{2}$ be the least positive integer such that $i_{k_{1}-1}\leq k_{2}q$. If $k_{2}=1$, take $M^{\prime}_{r-1}=X_{\gamma_{r-1}}X_{\gamma_{k_{1}-1}}$. Therefore, $M^{\prime}_{r-1}$ satisfies conditions $(1)$ and $(2)$.\\
If $k_{2}\geq 2$, proceeding recursively, there is a decreasing sequence
$$ r-1\geq k_{1}>k_{2}>\cdots>k_{l}=1$$
such that $i_{k_{j}-1}\leq k_{j+1}q$ for every $1\leq j\leq l-1$. Hence, we have $\alpha_{q}\leq \gamma_{k_{l-1}-1}$. \\
Now take $M^{\prime}_{r-1}=X_{\gamma_{r-1}}X_{\gamma_{k_{1}-1}}\cdots X_{\gamma_{k_{l-1}-1}}$. Then, by a similar argument as in the proof of \cref{Special lemma 1}, we can see $M^{\prime}_{r-1}$ satisfies conditions $(1)$ and $(2)$.
\end{proof}
In view of  \cref{sction:combinatorial resul:2} $X(w)^{ss}_{T_{J_{r}}}(\mathcal{L}(n\omega_{r}))\subseteq U_{w}wP/P$. The following Theorem describes the semistable locus.
\begin{theorem}\label{semistability criterion}
		Let $u\in U_{w}$. Then $uwP/P\in X(w)^{ss}_{T_{J_{r}}}(\mathcal{L}(n\omega_{r}))$ if and only if for every $1\leq j\leq r$, $X_{\beta_{i,j}}(u)\neq 0$ for some $i\in C_{j}$.
\end{theorem}

\begin{proof}
	($\Leftarrow:$) Let $x=uwP/P\in X(w)$. Assume that, for every $1\leq j\leq r$, there exists $i_{j}\in C_{j}$  such that $X_{\beta_{i_{j},j}}(u)\neq 0$. Fix such an $i_{j}\in C_{j}$.\\
		By arguments as in the \cref{Special lemma 1} and \cref{Special lemma 2}, for every $1\leq j\leq r-2$, there is a monomial $M^{\prime}_{j}$ in $X_{\gamma_{i}}$'s $(1\leq i\leq j)$ such that $M^{\prime}_{j}\in \psi_{2}^{*}(H^{0}(X(w),\mathcal{L}(\omega_{r})))$ and the weight $\mu^{\prime}_{j}$ of $M^{\prime}_{j}$ satisfies
		\begin{align}\label{Special eq 3}
			\langle\mu^{\prime}_{j},\lambda_{lq}\rangle&=\begin{dcases}
				-1&\text{ for $1\leq l\leq j$}\\
				0&\text{ for $j+1\leq l\leq r$}. 
			\end{dcases}
		\end{align}
	Let $M_{1}=M_{r}(\prod_{j=1}^{r-1}M^{\prime}_{j})$. Therefore, by \cref{Special lemma 1}, \cref{Special lemma 2}, and by \cref{Special eq 3}, we have $M_{1}\in\psi_{2}^{*}(H^{0}(X(w),\mathcal{L}((n-1)\omega_{r})))$ and the weight $\mu$ of $M_{1}$ satisfies 
	\begin{align}\label{Special eq 1}
		\langle\mu,\lambda_{jq}\rangle=-(n-j) \text{for every $1\leq j\leq r$}
	\end{align}
From \cref{remark 1.2}, we have for every $\tau\leq w$ in $W^{S\setminus\{\alpha_{r}\}}$, the weight of $p_{\tau}$ is equal to the weight of $\psi_{2}^{*}(p_{\tau})-w(\omega_{r})$. Let $s_{1}\in H^{0}(X(w),\mathcal{L}((n-1)\omega_{r}))$ be such that $\psi_{2}^{*}(s_{1})=M_{1}$. Hence by \cref{Lemma 4.1} and \cref{Special eq 1}, $(s_{1}\cdot p_{w})\in H^{0}(X(w),\mathcal{L}(n\omega_{r}))^{T_{J_{r}}}$ such that $(s_{1}\cdot p_{w})(x)\neq 0$. Therefore, $x\in X(w)^{ss}_{T_{J_{r}}}(\mathcal{L}(n\omega_{r}))$.
(:$\Rightarrow$)	Let $u\in U_{w}$ be such that $x=uwP/P\in X(w)^{ss}_{T_{J_{r}}}(\mathcal{L}(n\omega_{r}))$. Assume on the contrary that, there exists $j\in\{1,2,\ldots,r\}$ such that $X_{\alpha_{i,j}}(u)=0$ for every $i\in\{j,j+1,\ldots,jq\}$ (see \cref{Special eq 2}).\\
	{\it{Case 1}}: $j=r$. Let $w_{1}=(s_{q}\cdots s_{1})\cdots (s_{(r-1)q}\cdots s_{r-1})$. Then $w=w_{1}(s_{rq}\cdots s_{r})$ and $R^{+}(w^{-1})=R^{+}(w_{1}^{-1})\sqcup w_{1}R^{+}((s_{rq}\cdots s_{r})^{-1})$. Therefore, $x\in \prod_{\beta\in R^{+}(w_{1}^{-1})}U_{\beta}wP/P$, by hypothesis. Hence, $w_{1}^{-1}x\in U^{-}(s_{rq}\cdots s_{r})P/P$. Since $w_{1}^{-1}\lambda_{rq}=\lambda_{rq}$, using \cref{variation of Seshadri's lemma}, we have 
	\begin{align}\label{Special eq 5}
		\mu^{\mathcal{L}(n\omega_{r})}(x,-\lambda_{rq})=\langle (s_{rq}\cdots s_{r})(n\omega_{r}),\lambda_{rq}\rangle.
	\end{align}
From \cite[Lemma 5.8]{GKgit}, $X(s_{rq}\cdots s_{r})$, is the unique minimal dimensional Schubert variety in $G/P$ admitting semistable points for $\lambda_{rq}(\mathbb{G}_{m})$-linearized line bundle $\mathcal{L}(n\omega_{r})$. Therefore, from \cite[Lemma 5.4]{GKgit}, we have $\langle (s_{rq}\cdots s_{r})(n\omega_{r}),\lambda_{rq}\rangle<0$. Thus from \cref{Special eq 5} and \cref{HM theorem}, we have that $x\notin X(w)^{ss}_{T_{J_{r}}}(\mathcal{L}(n\omega_{r}))$.

{\it{Case 2}}:  Let $1\leq j\leq r-1$. We prove that there exists a $\lambda\in Y(T_{J_{r}})$ such that $\mu^{\mathcal{L}(n\omega_{r})}(x,\lambda)<0$.\\ Let \begin{align}\label{New proof:1}
	u=\displaystyle\prod_{\substack{l=1\\l\neq j}}^{r}\prod_{i=l}^{lq}u_{\alpha_{i,l}}(a_{i,l}), \text{where $a_{i,l}\in\mathbb{C}$, for all $l\neq j$ and $i=l,\ldots,lq$}.
\end{align}
Following the idea of the proof from \cite[Lemma 5.1]{Ses2}, we find a $\lambda\in Y(T_{J_{r}})$ such that $\lim_{t \to 0} \lambda(t)x=u_{1}wP/P$, where $u_{1}$ centralizes $\lambda$.\\
Let $\lambda=-(n-1)\lambda_{jq}+n\lambda_{(j+1)q}$. 

Let $l\in\{1,2,\ldots,r\}\setminus\{j\}$ and $i\in\{l,l+1,\ldots,lq\}$. Since $l\neq j$, using the description of $\alpha_{i,l}$, we can see that if $\alpha_{jq}\leq \alpha_{i,l}$, then $\alpha_{(j+1)q}\leq \alpha_{i,l}$ and in this case $\langle \alpha_{i,l},\lambda\rangle=1$. If $\alpha_{jq}\nleq\alpha_{i,l}$, then $\langle \alpha_{i,l},\lambda\rangle=n$ or $0$ depending on whether $\alpha_{(j+1)q}\leq \alpha_{i,l}$ or $\alpha_{(j+1)q}\nleq \alpha_{i,l}$ respectively.
Thus for every $l\in\{1,2,\ldots,r\}\setminus\{j\}$ and $i\in\{l,l+1,\ldots,lq\}$, $\langle\alpha_{i,l},\lambda\rangle\in\{0,1,n\}$. Thus, by \cref{New proof:1} we have 
\begin{align*}
	\lim_{t \to 0} \lambda(t)x=\prod_{\substack{\alpha_{i,l}\in R^{+}(w^{-1})\\ \langle\alpha_{i,l},\lambda\rangle=0}}u_{\alpha_{i,l}}(a_{i,l})wP/P.
\end{align*}
Let $u_{1}=\prod_{\substack{\alpha_{i,l}\in R^{+}(w^{-1})\\ \langle\alpha_{i,l},\lambda\rangle=0}}u_{\alpha_{i,l}}(a_{i,l})$. Then $u_{1}$ centralizes $\lambda$. Therefore, using \cite[Proposition 2.1]{Ses2}, we have \begin{align*}
	\mu^{\mathcal{L}(n\omega_{r})}(x,\lambda)&=\mu^{\mathcal{L}(n\omega_{r})}(\lim_{t \to 0} \lambda(t)x,\lambda)
	=	\mu^{\mathcal{L}(n\omega_{r})}(u_{1}wP/P,\lambda).
\end{align*}
Now from \cite[Proposition 3.1]{Ses2}, we have
	\begin{align*}
		\mu^{\mathcal{L}(n\omega_{r})}(u_{1}wP/P,\lambda)
=\mu^{\mathcal{L}(n\omega_{r})}(wP/P,\lambda).
	\end{align*}
Using the similar computation as in \cite[Lemma 5.1]{Ses2}, we can see that the $\lambda(\mathbb{G}_{m})$ acts with weight $-\langle w(n\omega_{r}),\lambda\rangle$ on the fiber of the $T$-fixed point $wP/P$. The sign here is negative because we use the left action of $G$ on $G/B$. Hence from \cite[Definition 2.2]{Ses2}, we have \begin{align*}
	\mu^{\mathcal{L}(n\omega_{r})}(wP/P,\lambda)=-\langle w(n\omega_{r}),\lambda\rangle.
\end{align*}
From \cref{Lemma 4.1}, we can see $-\langle w(n\omega_{r}),\lambda\rangle=-j$. Hence, $\mu^{\mathcal{L}(n\omega_{r})}(wP/P,\lambda)=-j<0$. Now from \cref{HM theorem}, $x$ is not a semistable point.
\end{proof}
\begin{lemma}\label{semistability=stability}
\begin{enumerate}
		\item We have $X(w)^{ss}_{T_{J_{r}}}(\mathcal{L}(n\omega_{r}))=X(w)^{s}_{T_{J_{r}}}(\mathcal{L}(n\omega_{r}))$.
		\item  The quotient $\GmodX{T_{J_{r}}}{X(w)}{\mathcal{L}(n\omega_{r})}$ is smooth.
\end{enumerate}
\end{lemma}
\begin{proof}
	{\it Claim}: For every $x\in X(w)^{ss}_{T_{J_{r}}}(\mathcal{L}(n\omega_{r}))$, the stabilizer $(T_{J_{r}})_{x}$ of $x$ in $T_{J_{r}}$ is trivial.\\
	{\it Proof of claim}: Let $x\in X(w)^{ss}_{T_{J_{r}}}(\mathcal{L}(n\omega_{r}))$. Let $x=uwP/P$, where $u=\prod_{\substack{1\leq j\leq r\\i\in C_{j}}}u_{\beta_{i,j}}(a_{i,j})$, $a_{i,j}\in\mathbb{C}$. Then by \cref{semistability criterion}, we have for all $1\leq j\leq r$, there exists $i_{j}\in C_{j}$ such that $a_{i_{j},j}\neq 0$. Let $t\in (T_{J_{r}})_{x}$. Then we have 
	\begin{align*}
		tx=\prod_{\substack{1\leq j\leq r\\i\in C_{j}}}u_{\beta_{i,j}}(\beta_{i,j}(t)a_{i,j})wP/P
	\end{align*} 
	Since $tx=x$ and $a_{i_{j},j}\neq 0$ for all $1\leq j\leq r$, we have  \begin{align}\label{section 4:eq 1.1}
		\beta_{i_{j},j}(t)=1
	\end{align}
Le $J:=\{1,\ldots,n-1\}\setminus\{jq:1\leq j\leq r\}$. Then, we have \begin{align}\label{section 4:eq 1.2}
		T_{J_{r}}\subseteq\cap_{k\in J}ker\ \alpha_{k}
	\end{align}
	Therefore, from \cref{section 4:eq 1.1} and \cref{section 4:eq 1.2}, we have 
	\begin{align}\label{section 4:eq 1.3}
		t\in (\cap_{j=1}^{r}ker\ \beta_{i_{j},j})\cap(\cap_{k\in J}ker\ \alpha_{k})
	\end{align}
{\it Subclaim}: The set $\{\beta_{i_{j},j}:1\leq j\leq r\}\cup\{\alpha_{k}:k\in J\}$ is linearly independent.\\
{\it Proof of subclaim}: Assume that \begin{align}\label{section 4:eq 1.4}
	 \sum_{j=1}^{r}c_{j}\beta_{i_{j},j}+\sum_{k\in J}d_{k}\alpha_{k}=0
\end{align}where $c_{j},d_{k}\in\mathbb{Z}$ for all $j\in\{1,\ldots,r\}$ and $k\in J$.

We first prove that $c_{j}=0$ for all $1\leq j\leq r$. 
Assume on the contrary that $c_{j}\neq 0$ for some $j\in\{1,\ldots,r\}$. Let $j_{0}$ be the largest $j\in \{1,\ldots,r\}$ such that $c_{j_{0}}\neq 0$. From the definition of $\beta_{i_{j},j}$'s we have $\alpha_{j_{0}q}\nleq\beta_{i_{j},j}$ for all $1\leq j\leq j_{0}-1$. Further, $\alpha_{j_{0}q}\leq \beta_{i_{j_{0}},j_{0}}$. Since $\alpha_{j_{0}q}$ is a cominuscule simple root, we have $\langle\sum_{j=1}^{r}c_{j}\beta_{i_{j},j},\lambda_{j_{0}q}\rangle=c_{j_{0}}$. Since $j_{0}q\notin J$, we have $\langle \alpha_{k},\lambda_{j_{0}q}\rangle=0$ for all $k\in J$. Therefore from \cref{section 4:eq 1.4}, we have \begin{align*}
	0=\sum_{j=1}^{r}c_{j}\langle\beta_{i_{j},j},\lambda_{j_{0}q}\rangle+\sum_{k\in J}d_{k}\langle\alpha_{k},\lambda_{j_{0}q}\rangle=c_{j_{0}}.
\end{align*}
This gives a contradiction to the assumption that $c_{j_{0}}\neq 0$. Thus we have $c_{j}=0$ for all $1\leq j\leq r$.\\
 Therefore, from \cref{section 4:eq 1.4}, we have $\sum_{k\in J}d_{k}\alpha_{k}=0$. Since $\{\alpha_{k}:k\in J\}$ is linearly independent, we have $d_{k}=0$ for all $k\in J$. This proves the subclaim.

Therefore, the rank of $\sum_{j=1}^{r}\mathbb{Z}\beta_{i_{j},j}+\sum_{k\in J}\mathbb{Z}\alpha_{k}$ is $n-1$. Further, we have  $$\sum_{j=1}^{r}\mathbb{Z}\beta_{i_{j},j}+\sum_{k\in J}\mathbb{Z}\alpha_{k}\subseteq\sum_{l=1}^{n-1}\mathbb{Z}\alpha_{l}.$$
Hence, $\sum_{j=1}^{r}\mathbb{Q}\beta_{i_{j},j}+\displaystyle\sum_{k\in J}\mathbb{Q}\alpha_{k}=\sum_{l=1}^{n-1}\mathbb{Q}\alpha_{l}$. Therefore,  from \cite[lemma 3.2, page 193]{Kannan 3}, we have $\sum_{j=1}^{r}\mathbb{Z}\beta_{i_{j},j}+\displaystyle\sum_{k\in J}\mathbb{Z}\alpha_{k}=\sum_{l=1}^{n-1}\mathbb{Z}\alpha_{l}$. Thus , we have \begin{align*}
	(\cap_{j=1}^{r}ker\ \beta_{i_{j},j})\cap(\cap_{k\in J}ker\ \alpha_{k})=\cap_{l=1}^{n-1} ker\ \alpha_{l}.
\end{align*}
Hence from \cref{section 4:eq 1.3}, we have $t\in \cap_{l=1}^{n-1}ker\ \alpha_{l}$. We have $\cap_{l=1}^{n-1}ker\ \alpha_{l}=Z(PSL(n,\mathbb{C}))$. Since $Z(PSL(n,\mathbb{C}))=\{id\}$, we have $t=id$. This completes the proof of the claim.

{\it Proof of $(1)$}: From the above claim,
 for every $x\in 	X(w)^{ss}_{T_{J_{r}}}(\mathcal{L}(n\omega_{r}))$, we have that $dim(T_{J_{r}}\cdot x)=dim (T_{J_{r}})$. Thus every $T_{J_{r}}$ orbit in $X(w)^{ss}_{T_{J_{r}}}(\mathcal{L}(n\omega_{r}))$ is of minimal dimension and hence closed (see \cite[Proposition 8.3]{Hum2}). Again from claim, we have $(T_{J_{r}})_{x}$ is trivial for all  $x\in 	X(w)^{ss}_{T_{J_{r}}}(\mathcal{L}(n\omega_{r}))$. Thus we can conclude that every 
	$x\in 	X(w)^{ss}_{T_{J_{r}}}(\mathcal{L}(n\omega_{r}))$ is a stable point.\\
	{\it Proof of $(2)$}: From \cref{sction:combinatorial resul:2}, we have $X(w)$ is a minimal dimensional Schubert variety admitting semistable points for the $T_{J_{r}}$-linearized line bundle $\mathcal{L}(n\omega_{r})$. Hence from Bruhat decomposition, it follows that $X(w)^{ss}_{T_{J_{r}}}(\mathcal{L}(n\omega_{r}))\subseteq BwP/P$. Thus $X(w)^{ss}_{T_{J_{r}}}(\mathcal{L}(n\omega_{r}))$ is smooth open subset of $X(w)$. From the above claim, we have the action of $T_{J_{r}}$ on $X(w)^{ss}_{T_{J_{r}}}(\mathcal{L}(n\omega_{r}))$ is free. From part (1), we have the quotient is geometric. Then the result follws from Luna's slice theorem (see \cite[Proposition 5.7]{Luna}).
\end{proof}
\section{Description of  $\GmodX{T_{J_{r}}}{X(w_{r,n})}{\mathcal{L}(n\omega_{r})}$ for $n\equiv 1 (mod \ r)$}\label{section 5}
Assume that $r,q\in\mathbb{N}$ be such that $r\geq 2$ and $q\geq 2$. Let $n_{r}=n=rq+1$ and $n_{r-1}=(r-1)q+1$. In this section, we construct a morphism $\phi_{r}:Y_{r}\longrightarrow Y_{r-1}$ such that  $Y_{r}$ is a $\mathbb{P}^{r(q-1)}$-bundle over $Y_{r-1}$.

Let $G^{\prime}=PSL(n_{r-1},\mathbb{C})$. Let $T^{\prime}$ be a maximal torus of $G^{\prime}$ and $B^{\prime}$ be a Borel subgroup of $G^{\prime}$ containing $T^{\prime}$. The set of roots of $G^{\prime}$ with respect to $T^{\prime}$ will be denoted by $R^{\prime}$. Let $S^{\prime}:=\{\alpha_{1}^{\prime},\alpha_{2}^{\prime},\cdots,\alpha_{n_{r-1}-1}^{\prime}\}$ be the set of simple roots of $G^{\prime}$ with respect to $(B^{\prime},T^{\prime})$. Let $\{\omega^{\prime}_{1},\cdots,\omega^{\prime}_{n_{r-1}-1}\}$ denote the fundamental weights associated to $S^{\prime}$. Let $\{\lambda_{1}^{\prime},\ldots,\lambda_{n_{r-1}-1}^{\prime}\}$ denote the one parameter subgroups of $T^{\prime}$ dual to $\{\alpha_{1}^{\prime},\alpha_{2}^{\prime},\cdots,\alpha_{n_{r-1}-1}^{\prime}\}$. Let $Q$ denote the maximal parabolic subgroup of $G^{\prime}$ corresponding to simple root $\alpha_{r-1}^{\prime}$. Let $W^{\prime}=N_{G^{\prime}}(T^{\prime})/T^{\prime}$ denote the Weyl group of $G^{\prime}$ with respect to $T^{\prime}$. Let $s^{\prime}_{i}$ denote the simple reflection in $W^{\prime}$ corresponding to simple root $\alpha^{\prime}_{i}$. Let $X(w_{r-1,n_{r-1}})\subseteq G^{\prime}/Q$ be the unique minimal dimensional Schubert variety admitting semistable points for the $T^{\prime}$- linearized line bundle $\mathcal{L}(n_{r-1}\omega^{\prime}_{r-1})$. Then from \cite[Lemma 2.7]{sskannan}, a reduced expression of $w_{r-1,n_{r-1}}$ is given by
\begin{align}\label{eq 5.1.2}
	w_{r-1,n_{r-1}}=(s^{\prime}_{q}\cdots s^{\prime}_{1})(s^{\prime}_{2q}\cdots s^{\prime}_{2})\cdots (s^{\prime}_{(r-1)q}\cdots s^{\prime}_{r-1})
\end{align}
Let $T^{\prime}_{J_{r-1}}$ denote the subgroup of $T^{\prime}$ generated by $\lambda_{jq}^{\prime}(\mathbb{G}_{m})$ ($1\leq j\leq r-1$). 

For the simplicity of notation, we denote $w_{r,n_{r}}$ and $w_{r-1,n_{r-1}}$ by $w_{r}$ and $w_{r-1}$ respectively.

Define $Y_{r}:=\GmodX{T_{J_{r}}}{X(w_{r})}{\mathcal{L}(n_{r}\omega_{r})}$ and $Y_{r-1}:=\GmodX{T^{\prime}_{J_{r-1}}}{X(w_{r-1})}{\mathcal{L}(n_{r-1}\omega_{r-1}^{\prime})}$.\\
Recall that for $1\leq j\leq r$ and $i\in C_{j}$, $\beta_{i,j}:=\sum_{k=i}^{jq}\alpha_{k}$.
For every $1\leq j\leq r-1$ and  $i\in C_{j}$, let $\gamma_{i,j}:=\sum_{k=i}^{jq}\alpha^{\prime}_{k}$.
From \cite[Lemma 2.4]{sskannan}, we have \begin{align}
	R^{+}(w_{r}^{-1})&=\{\beta_{i,j}:1\leq j\leq r,i\in C_{j}\}\label{eq 5.1.3}\\
	R^{+}(w_{r-1}^{-1})&=\{\gamma_{i,j}:1\leq j\leq r-1, i\in C_{j}\}\label{eq 5.1.4}
\end{align}

For every $\beta\in R$ (respectively, $\gamma\in R^{\prime}$), let $u_{\beta}:\mathbb{C}\rightarrow U_{\beta}$ (respectively, $u^{\prime}_{\gamma}:\mathbb{C}\rightarrow U^{\prime}_{\gamma}$) be the isomorphism onto the roots subgroups such that $tu_{\beta}(a)t^{-1}=u_{\beta}(\beta(t)a)$ (respectively, $su^{\prime}_{\gamma}(a)s^{-1}=u^{\prime}_{\gamma}(\gamma(s)a)$)
 for all $t\in T$, $s\in T^{\prime}$ and $a\in\mathbb{C}$.

Let $U_{w_{r}}:=\prod_{\beta\in R^{+}(w_{r}^{-1})}U_{\beta}$ and $U_{w_{r-1}}:=\prod_{\gamma\in R^{+}(w_{r-1}^{-1})}U^{\prime}_{\gamma}$.

Let $V_{r}:=X(w_{r})^{ss}_{T_{J_{r}}}(\mathcal{L}(n_{r}\omega_{r}))$ and $V_{r-1}:=X(w_{r-1})^{ss}_{T^{\prime}_{J_{r-1}}}(\mathcal{L}(n_{r-1}\omega_{r-1}^{\prime}))$. 
Then from \cref{semistability criterion}, we have 
\begin{align}
	V_{r}&=\{uw_{r}P/P\in C(w_{r}): \text{for all $1\leq j\leq r$, $\exists i\in C_{j}$ such that} X_{\beta_{i,j}}(u)\neq 0 \}\label{eq 5.1.5}\\
	V_{r-1}=\{&u^{\prime}w_{r-1}Q/Q\in C(w_{r-1}): \text{for all $1\leq j\leq r-1$,  $\exists i\in C_{j}$ }
		\text{such that} X_{\gamma_{i,j}}(u^{\prime})\neq 0 \}\label{eq 5.1.6}
\end{align} 
In view of \cref{semistability=stability} (1), $Y_{r}$ and $Y_{r-1}$ are the geometric quotients $T_{J_{r}}\backslash V_{r}$ and $T^{\prime}_{J_{r-1}}\backslash V_{r-1}$ respectively.\\
Let $\pi_{r}:V_{r}\longrightarrow Y_{r}$ and $\pi_{r-1}:V_{r-1}\longrightarrow Y_{r-1}$ denote the GIT quotient maps.
\begin{lemma}\label{section:quotient:lemma 1}
	Let $1\leq j\leq r$. Let $i\in C_{j}$ and $(i,j)\in J_{d(i,j),j}$. \begin{enumerate}
		\item	If $1\leq j\leq r-1$. Then for every $1\leq l\leq r-1$, we have 
		\begin{align*}
			\langle \gamma_{i,j},\lambda^{\prime}_{lq}\rangle=&\begin{dcases}
				1&\text{if $d(i,j)\leq l\leq j$}\\
				0&\text{otherwise}
			\end{dcases}
		\end{align*}
		\item If $1\leq j\leq r$. Then for every $1\leq l\leq r$, we have 
		\begin{align*}
			\langle\beta_{i,j},\lambda_{lq}\rangle=&\begin{dcases}
				1&\text{ if $d(i,j)\leq l\leq j$}\\
				0&\text{ otherwise}
			\end{dcases}
		\end{align*}
	\end{enumerate}
\end{lemma}
\begin{proof}Proof is similar to \cref{claim 1}.
\end{proof}
\begin{lemma}\label{section:quotient:lemma 7}Let $J=(i_{1},i_{2},\cdots,i_{r})\in\Pi_{j=1}^{r}C_{j}$. 
	\begin{enumerate}
		\item 
		For $1\leq j\leq r-1$, 
		let $\delta_{J(j)}:=\sum_{k=0}^{m(i_{j},j)-1}\gamma_{i_{e^{k}(i_{j},j)},e^{k}(i_{j},j)}$. Then we have 
		\begin{align*}
			\langle\delta_{J(j)},\lambda^{\prime}_{lq}\rangle=\begin{dcases}
				1& \text{ if $1\leq l\leq j$}\\
				0& \text{ if $j+1\leq l\leq r-1$}
			\end{dcases}
		\end{align*}
		\item 	For $1\leq j\leq r$, 
		let $\gamma_{J(j)}:=\sum_{k=0}^{m(i_{j},j)-1}\beta_{i_{e^{k}(i_{j},j)},e^{k}(i_{j},j)}$. Then we have 
		\begin{align*}
			\langle\gamma_{J(j)},\lambda_{lq}\rangle=\begin{dcases}
				1& \text{ if $1\leq l\leq j$}\\
				0& \text{ if $j+1\leq l\leq r$}
			\end{dcases}
		\end{align*}
\text{$($Note that $\gamma_{J(j)}=\gamma_{j}$ as defined in \cref{eq 1.14}$)$}
	\end{enumerate}
\end{lemma}
\begin{proof} Proof is similar to \cref{claim 4}. 
\end{proof}
\begin{lemma}\label{section:quotient lemma 3}
	Let $t=\prod_{l=1}^{r}\lambda_{lq}(t_{l})$, where $t_{l}\in \mathbb{G}_{m}$ for $1\leq l\leq r$. Define $t^{\prime}:=\prod_{l=1}^{r-1}\lambda^{\prime}_{lq}(t_{l})$. Then for every $1\leq j\leq r-1$ and $i\in C_{j}$, we have $\beta_{i,j}(t)=\gamma_{i,j}(t^{\prime})$.
\end{lemma}
\begin{proof}
	Let $j\in\{1,\ldots, r-1\}$ and $i\in C_{j}$. We have $(i,j)\in J_{d(i,j),j}$. From part $(2)$ of \cref{section:quotient:lemma 1}, we have $\beta_{i,j}(t)=\prod_{ l=1}^{r}t_{l}^{\langle\beta_{i,j},\lambda_{lq}\rangle}=\prod_{l=d(i,j)}^{j}t_{l}$. From part (1) of \cref{section:quotient:lemma 1}, we have $\gamma_{i,j}(t^{\prime})=\prod_{l=1}^{r-1}t_{l}^{\langle\gamma_{i,j},\lambda^{\prime}\rangle}=\prod_{l=d(i,j)}^{j}t_{l}$.
	Therefore, we have $\beta_{i,j}(t)=\gamma_{i,j}(t^{\prime})$. 
\end{proof}
\subsection{Construction of $\phi_{r}:Y_{r}\longrightarrow Y_{r-1}$}\label{section 5.1} For $u=\displaystyle\prod_{\substack{1\leq j\leq r\\i\in C_{j}}}u_{\beta_{i,j}}(a_{i,j})\in U_{w_{r}}$, where $a_{i,j}\in\mathbb{C}$, define \begin{align}\label{eq 5.1.38}
	u[r-1]:=\prod_{\substack{1\leq j\leq r-1\\i\in C_{j}}}u^{\prime}_{\gamma_{i,j}}(a_{i,j}).
\end{align}
Note that for $u\in U_{w_{r}}$, we have $u[r-1]\in U_{w_{r-1}}$.
In view of \cref{remark 1.1}, we have a well defined morphism $f:C(w_{r})\longrightarrow C(w_{r-1})$ given by $f(uw_{r}P/P)=u[r-1]w_{r-1}Q/Q$ for all $u\in U_{w_{r}}$.\\
For any $x=uw_{r}P/P\in U_{w_{r}}w_{r}P/P$, define \begin{align}\label{eq 5.1.18}
	x[r-1]:=u[r-1]w_{r-1}Q/Q
\end{align}
  \begin{lemma}\label{eq 5.1.11} Let $u\in U_{w_{r}}$. For every $1\leq j\leq r-1$ and $i\in C_{j}$, we have
	\begin{align*}
		X_{\gamma_{i,j}}(u[r-1])=X_{\beta_{i,j}}(u).
	\end{align*}
\end{lemma}
\begin{proof}
 Let $u=\prod_{\substack{1\leq j\leq r\\i\in C_{j}}}u_{\beta_{i,j}}(a_{i,j})\in U_{w_{r}}$, where $a_{i,j}\in\mathbb{C}$. Then $X_{\beta_{i,j}}(u)=a_{i,j}$. From \cref{eq 5.1.38}, we have $X_{\gamma_{i,j}}(u[r-1])=a_{i,j}$. Therefore, we have $X_{\gamma_{i,j}}(u[r-1])=X_{\beta_{i,j}}(u)$.
\end{proof}

\begin{lemma}\label{section: quotient lemma 4}
	Let $x\in V_{r}$. Then, we have $x[r-1]\in V_{r-1}$. 
\end{lemma}
\begin{proof}
	Let $x=uw_{r}P/P\in V_{r}$. Let $u=\prod_{\substack{1\leq j\leq r\\i\in C_{j}}}u_{\beta_{i,j}}(a_{i,j})$, where $a_{i,j}\in\mathbb{C}$. From \cref{eq 5.1.5}, for every $1\leq j\leq r$, there is an integer $i\in C_{j}$ such that $a_{i,j}=X_{\beta_{i,j}}(u)\neq 0$. Therefore, using \cref{eq 5.1.11}, for every $1\leq j\leq r-1$, we have $X_{\gamma_{i,j}}(u[r-1])=a_{i,j}\neq 0$. Hence, from description of $V_{r-1}$ in \cref{eq 5.1.6}, we have $x[r-1]=u[r-1]w_{r-1}Q/Q\in V_{r-1}$.
\end{proof}
Recall $\pi_{r-1}:V_{r-1}\longrightarrow Y_{r-1}$ is the GIT quotient map.
Define a morphism $\phi:V_{r}\longrightarrow Y_{r-1}$ by  \begin{align}\label{eq 5.1.8}
	\phi(x)=\pi_{r-1}(x[r-1]),\ \text{for all $x\in V_{r}$}.
\end{align}
\begin{lemma}\label{section:quotient:lemma 2}
	The morphism $\phi$ is constant on $T_{J_{r}}$-orbits.
\end{lemma}
\begin{proof}
	Let $t\in T_{J_{r}}$. Let  $t=\prod_{l=1}^{r}\lambda_{lq}(t_{l})$, where $t_{l}\in\mathbb{G}_{m}$ for all $1\leq l\leq r$. Define $t^{\prime}:=\prod_{l=1}^{r-1}\lambda^{\prime}_{lq}(t_{l})$.\\
	Let $x=uw_{r}P/P\in V_{r}$, where 
	$u=\prod_{\substack{1\leq j\leq r\\i\in C_{j}}}u_{\beta_{i,j}}(a_{i,j})$, $a_{i,j}\in\mathbb{C}$. Then, $$tx=\prod_{\substack{1\leq j\leq r\\i\in C_{j}}}u_{\beta_{i,j}}(\beta_{i,j}(t)a_{i,j})w_{r}P/P.$$
	Hence, $(tx)[r-1]=\prod_{\substack{1\leq j\leq r-1\\i\in C_{j}}}u^{\prime}_{\gamma_{i,j}}(\beta_{i,j}(t)a_{i,j})w_{r-1}Q/Q$.
	Therefore, from \cref{section:quotient lemma 3}, we have $(tx)[r-1]=\prod_{\substack{1\leq j\leq r-1\\i\in C_{j}}}u^{\prime}_{\gamma_{i,j}}(\gamma_{i,j}(t^{\prime})a_{i,j})w_{r-1}Q/Q$. Hence,
	\begin{align*}
		(tx)[r-1]&=t^{\prime}(\prod_{\substack{1\leq j\leq r-1\\i\in C_{j}}}u^{\prime}_{\gamma_{i,j}}(a_{i,j})w_{r-1}Q/Q)\\
		&=t^{\prime}(x[r-1]).
	\end{align*}
	Therefore, $\phi(tx)=\pi_{r-1}(t^{\prime}(x[r-1]))=\pi_{r-1}(x[r-1])=\phi(x)$.	
	\end{proof}
	Therefore, in view of \cref{section:quotient:lemma 2}, $\phi$ induces a morphism $\phi_{r}:Y_{r}\longrightarrow Y_{r-1}$ such that $\phi=\phi_{r}\circ\pi_{r}$.
	\subsection{Open covering of $Y_{r-1}$ }
	For every $J=(i_{1},i_{2},\cdots,i_{r-1})\in \Pi_{j=1}^{r-1}C_{j}$, we define
	\begin{align}\label{eq 5.1.9}
		\tilde{U}(J):=\{u^{\prime}w_{r-1}Q/Q\in V_{r-1}: X_{\gamma_{i_{j},j}}(u^{\prime})\neq 0, \text{for all } 1\leq j\leq r-1\}
	\end{align} 
	\begin{align}\label{eq 5.1.10}
		\tilde{V}(J):=\{uw_{r}P/P\in V_{r}: X_{\beta_{i_{j},j}}(u)\neq 0,\ \text{for all } 1\leq j\leq r-1\}
	\end{align}
	\begin{lemma}\label{section: quotient lemma 5}
		Let	$J=(i_{1},i_{2},\cdots,i_{r-1})\in \Pi_{j=1}^{r-1}C_{j}$. Then, we have 
		\begin{enumerate}
			\item $\tilde{U}(J)$ is stable under the action of $T^{\prime}_{J_{r-1}}$.
			\item $\tilde{V}(J)$ is stable under the action of $T_{J_{r}}$.
		\end{enumerate}
	\end{lemma}
	\begin{proof}
		{\it Proof of $(1)$}: Let $x=u^{\prime}w_{r-1}Q/Q\in \tilde{U}(J)$ and $t^{\prime}\in T^{\prime}_{J_{r-1}}$. We have $t^{\prime}x=t^{\prime}u^{\prime}w_{r-1}Q/Q$, and  $X_{\gamma_{i_{j},j}}(t^{\prime}u^{\prime})=\gamma_{i_{j},j}(t^{\prime})X_{\gamma_{i_{j},j}}(u^{\prime})$. 
		Since $x\in \tilde{U}(J)$, from \cref{eq 5.1.9}, we have 
		$X_{\gamma_{i_{j},j}}(u^{\prime})\neq 0$ for all $1\leq j\leq r-1$. Therefore, $X_{\gamma_{i_{j},j}}(t^{\prime}u^{\prime})\neq 0$. Hence, $t^{\prime}x\in \tilde{U}(J)$.\\
		{\it Proof of $(2)$}: Let $x=uw_{r}P/P\in \tilde{V}(J)$ and $t\in T_{J_{r}}$. 
		Then, we have $tx=tuw_{r}P/P$ and  $X_{\beta_{i_{j},j}}(tu)=\beta_{i_{j},j}(t)X_{\beta_{i_{j},j}}(u)$. Since $x\in \tilde{V}(J)$, we have $X_{\beta_{i_{j},j}}(u)\neq 0$ for all $1\leq j\leq r-1$ (see \cref{eq 5.1.10}). Therefore, $X_{\beta_{i_{j},j}}(tu)\neq 0$. Hence, $tu\in \tilde{V}(J)$.
	\end{proof}
	For every $J\in \Pi_{j=1}^{r-1}C_{j}$, we define
	\begin{align}
		U(J):&=\pi_{r-1}(\tilde{U}(J))\label{eq 5.1.35}\\ \text{\ and\ } 	V(J):&=\pi_{r}(\tilde{V}(J))\label{eq 5.1.44}
	\end{align}	
	Since $\tilde{U}(J)$ is $T^{\prime}_{J_{r-1}}$-stable open subvariety of $X(w_{r-1})$, $U(J)$ is open in $Y_{r-1}$.
	From \cref{semistability criterion}, we have $V_{r-1}=\displaystyle\bigcup_{J\in\Pi_{j=1}^{r-1}C_{j}}\tilde{U}(J)$. Hence, $Y_{r-1}=\displaystyle\bigcup_{J\in\Pi_{j=1}^{r-1}C_{j}}U(J)$  is a covering of $Y_{r-1}$ by open subsets.
	\begin{lemma} \label{section: quotient lemma 6}Let $J\in\Pi_{j=1}^{r-1}C_{j}$. Then we have 
	\begin{enumerate}
			\item  $\pi_{r-1}^{-1}(U(J))=\tilde{U}(J)$.
			\item   $\pi_{r}^{-1}(V(J))=\tilde{V}(J)$.
	\end{enumerate}
	\end{lemma}
	\begin{proof}{\it Proof of (1)}:
		Let $y\in V_{r-1}$ be such that $\pi_{r-1}(y)\in U(J)$. Therefore, $\pi_{r-1}(y)=\pi_{r-1}(y^{\prime})$ for some $y^{\prime}\in \tilde{U}(J)$. Since $X(w_{r-1})^{ss}_{T^{\prime}_{J_{r-1}}}(\mathcal{L}(n_{r-1}\omega^{\prime}_{r-1}))=X(w_{r-1})^{s}_{T^{\prime}_{J_{r-1}}}(\mathcal{L}(n_{r-1}\omega^{\prime}_{r-1}))$, we have $y=t^{\prime}y^{\prime}$ for some $t^{\prime}\in T^{\prime}_{J_{r-1}}$. Hence from \cref{section: quotient lemma 5} $(1)$, we have $y\in \tilde{U}(J)$. Therefore, $\pi_{r-1}^{-1}(U(J))\subseteq\tilde{U}(J)$. Other inclusion follows from \cref{eq 5.1.35}. Thus we have $\pi_{r-1}^{-1}(U(J))=\tilde{U}(J)$.\\
	Proof of $(2)$ is similar to the proof of $(1)$.
	\end{proof}
	\begin{lemma}\label{section:quotient:lemma 10}
		\begin{enumerate}
			\item For any $J_{1},J_{2}\in\Pi_{j=1}^{r-1}C_{j}$, we have $\pi_{r-1}(\tilde{U}(J_{1})\cap \tilde{U}(J_{2}))=U(J_{1})\cap U(J_{2})$.
			\item  For any three $J_{1},J_{2},J_{3}\in\Pi_{j=1}^{r-1}C_{j}$, we have $\pi_{r-1}(\tilde{U}(J_{1})\cap \tilde{U}(J_{2})\cap \tilde{U}(J_{3}))=U(J_{1})\cap U(J_{2})\cap U(J_{3}).$
		\end{enumerate}
	\end{lemma}
	\begin{proof}{\it{Proof of $(1)$}}:
		Let $x\in U(J_{1})\cap U(J_{2})$. Therefore, $x=\pi_{r-1}(y_{1})$ and $x=\pi_{r-1}(y_{2})$ for some $y_{1}\in \tilde{U}(J_{1})$ and $y_{2}\in \tilde{U}(J_{2})$. Hence, $\pi_{r-1}(y_{1})=\pi_{r-1}(y_{2})$. This implies, $y_{2}=t^{\prime}y_{1}$ for some $t^{\prime}\in T^{\prime}_{J_{r-1}}$. Since $y_{1}\in \tilde{U}(J_{1})$, from \cref{section: quotient lemma 5}(1), we have  $y_{2}\in \tilde{U}(J_{1})$. Therefore, $x\in \pi_{r-1}(\tilde{U}(J_{1})\cap \tilde{U}(J_{2}))$. Hence, $U(J_{1})\cap U(J_{2})\subseteq \pi_{r-1}(\tilde{U}(J_{1})\cap \tilde{U}(J_{2}))$. Other inclusion follows from \cref{eq 5.1.35}.\\
		Proof of $(2)$ is similar to the proof of $(1)$.
	\end{proof}
\begin{lemma}\label{section:quotient:lemma 14}
	For any $J_{1},J_{2}\in\Pi_{j=1}^{r-1}C_{j}$, we have $\pi_{r}(\tilde{V}(J_{1})\cap\tilde{V}(J_{2}))=V(J_{1})\cap V(J_{2})$.
\end{lemma}
\begin{proof}
	Proof is similar to the proof of \cref{section:quotient:lemma 10}(1).
\end{proof}
	\begin{lemma}\label{section:quotient:lemma 13}
		For every $J\in\Pi_{j=1}^{r-1}C_{j}$, we have $\phi_{r}^{-1}(U(J))=V(J)$.
	\end{lemma}
	\begin{proof}Let $J=(i_{1},i_{2},\cdots,i_{r-1})$.
		Let $z\in Y_{r}$ be such that $\phi_{r}(z)\in U(J)$. Let $z=\pi_{r}(x)$, where $x=uw_{r}P/P\in V_{r}$. From \cref{eq 5.1.8}, we have $\phi_{r}(z)=\pi_{r-1}(x[r-1])$. Hence, $\pi_{r-1}(x[r-1])\in U(J)$. Now from \cref{section: quotient lemma 6}(1), we have $x[r-1]\in\tilde{U}(J)$. Therefore, $X_{\gamma_{i_{j},j}}(u[r-1])\neq 0$, for all $1\leq j\leq r-1$. From \cref{eq 5.1.11}, we have $X_{\beta_{i_{j},j}}(u)=X_{\gamma_{i_{j},j}}(u[r-1])\neq 0$, for all $1\leq j\leq r-1$. Hence, $x\in \tilde{V}(J)$. Thus, $z\in V(J)$. Therefore, \begin{align}\label{eq 5.1.36}
			\phi_{r}^{-1}(U(J))\subseteq V(J).
		\end{align}
		Let $z\in V(J)$. Let $z=\pi_{r}(x)$, for some $ x=uw_{r}P/P\in \tilde{V}(J)$. Then, we have $X_{\beta_{i_{j},j}}(u)\neq 0$ for all $1\leq j\leq r-1$. We have $\phi_{r}(z)=\pi_{r-1}(x[r-1])$. From \cref{eq 5.1.11}, we have $X_{\gamma_{i_{j},j}}(u[r-1])=X_{\beta_{i_{j},j}}(u)\neq 0$.  Therefore, $u[r-1]w_{r-1}Q/Q\in\tilde{U}(J)$ (see \cref{eq 5.1.9}). Hence,  $\phi_{r}(z)=\pi_{r-1}(x[r-1])=\pi_{r-1}(u[r-1]w_{r-1}Q/Q)\in U(J)$. Thus, \begin{align}\label{eq 5.1.37}
			V(J)\subseteq \phi_{r}^{-1}(U(J)).
		\end{align}
		Now the proof follows from \cref{eq 5.1.36} and \cref{eq 5.1.37}.
	\end{proof}
	\subsection{Construction of function $b(J(j)):\tilde{U}(J)\longrightarrow\mathbb{C^{\times}}$ for $0\leq j\leq r-1$}\ 
	\\
	Let $J=(i_{1},i_{2},\cdots,i_{r-1})\in\Pi_{j=1}^{r-1}C_{j}$. For every $1\leq j\leq r-1$, define $J(j):=(i_{1},\cdots, i_{j})$ and $J(0):=\phi.$\\ Recall the definition of $e^{k}(i_{j},j)$ and $m(i_{j},j)$ from \cref{construction of e^{k}} (see \cref{Special eq 6} and \cref{e^{m}=0}).
	Let $y=\tilde{u}w_{r-1}Q/Q\in\tilde{U}(J)$, where $\tilde{u}=\prod_{\substack{1\leq j\leq r-1\\i\in C_{j}}}u^{\prime}_{\gamma_{i,j}}(a_{i,j}), a_{i,j}\in\mathbb{C}$.
	
	\begin{align}
		\text{For $1\leq j\leq r-1$, define\ }&
		b(J(j))(y):=\prod_{k=0}^{m(i_{j},j)-1}X_{\gamma_{i_{e^{k}(i_{j},j)},e^{k}(i_{j},j)}}(\tilde{u})\label{eq 5.1.12}\\
		&b(J(0))(y):=1\label{eq 5.1.14}
	\end{align}

	Let $J_{1},J_{2}\in\Pi_{j=1}^{r-1}C_{j}$. For $0\leq j\leq r-1$, define the function
	$$\tilde{b}_{J_{1},J_{2}}(j):\tilde{U}(J_{1})\cap \tilde{U}(J_{2})\longrightarrow\mathbb{C^{\times}}$$ by \begin{align}\label{eq 5.1.15}
		\tilde{b}_{J_{1},J_{2}}(j)(y):=\frac{b(J_{1}(j))(y)}{b(J_{2}(j))(y)},\text{\ for all
			$y\in \tilde{U}(J_{1})\cap \tilde{U}(J_{2})$}
	\end{align}
	\begin{lemma}\label{section:quotient:lemma 8}
		Let $J\in\Pi_{j=1}^{r-1}C_{j}$. Then for every $t^{\prime}\in T^{\prime}_{J_{r-1}}$ and $1\leq j\leq r-1$,  we have $$b(J(j))(t^{\prime}y)=\delta_{J(j)}(t^{\prime})b(J(j))(y),\text{\ for all $y\in \tilde{U}(J)$}.$$
	\end{lemma}
	\begin{proof}
		Let	$y=\tilde{u}w_{r-1}Q/Q\in\tilde{U}(J)$. Then from \cref{section: quotient lemma 5}(1), we have, $t^{\prime}y=t^{\prime}\tilde{u}w_{r-1}Q/Q\in \tilde{U}(J)$. Therefore, from \cref{eq 5.1.12}, we have 
		\begin{align*}
			b(J(j))(t^{\prime}y)&=\prod_{k=0}^{m(i_{j},j)-1}X_{\gamma_{i_{e^{k}(i_{j},j)},e^{k}(i_{j},j)}}(t^{\prime}\tilde{u}).
		\end{align*}
		Since $X_{\gamma_{i_{e^{k}(i_{j},j)},e^{k}(i_{j},j)}}(t^{\prime}\tilde{u})=\gamma_{i_{e^{k}(i_{j},j)},e^{k}(i_{j},j)}(t^{\prime})X_{\gamma_{i_{e^{k}(i_{j},j)},e^{k}(i_{j},j)}}(\tilde{u})$, we have 
		\begin{align}\label{eq 5.1.13}
			b(J(j))(t^{\prime}y)&=(\prod_{k=0}^{m(i_{j},j)-1}\gamma_{i_{e^{k}(i_{j},j)},e^{k}(i_{j},j)}(t^{\prime}))(\prod_{k=0}^{m(i_{j},j)-1}X_{\gamma_{i_{e^{k}(i_{j},j)},e^{k}(i_{j},j)}}(\tilde{u}))
		\end{align}
		Now from \cref{section:quotient:lemma 7}$(1)$, we have $\prod_{k=0}^{m(i_{j},j)-1}\gamma_{i_{e^{k}(i_{j},j)},e^{k}(i_{j},j)}(t^{\prime})=\delta_{J(j)}(t^{\prime})$. Thus, from \cref{eq 5.1.13} and \cref{eq 5.1.12}, we have 	$b(J(j))(t^{\prime}y)=\delta_{J(j)}(t^{\prime})b(J(j))(y)$.
	\end{proof}
	\begin{lemma}\label{section:quotient:lemma 9}
		Let $J_{1},J_{2}\in\Pi_{j=1}^{r-1}C_{j}$. Then for all $0\leq j\leq r-1$, $\tilde{b}_{J_{1},J_{2}}(j)$ is constant on $T^{\prime}_{J_{r-1}}$-orbits.
	\end{lemma}
	\begin{proof}
		Let $t^{\prime}\in T^{\prime}_{J_{r-1}}$. Let $t^{\prime}=\prod_{l=1}^{r-1}\lambda^{\prime}_{lq}(t_{l})$, where $t_{l}\in \mathbb{G}_{m}$ for all $1\leq l\leq r-1$. Let $y\in \tilde{U}(J_{1})\cap \tilde{U}(J_{2})$. From \cref{section: quotient lemma 5} $(1)$, $\tilde{U}(J_{1})\cap \tilde{U}(J_{2})$ is $T^{\prime}_{J_{r-1}}$-stable. Hence, $t^{\prime}y\in \tilde{U}(J_{1})\cap \tilde{U}(J_{2})$. Now for, $j=0$ we have $\tilde{b}_{J_{1},J_{2}}(0)(t^{\prime}y)=\frac{b(J_{1}(0))(t^{\prime}y)}{b(j_{2}(0))(t^{\prime}y)}=1$ (see \cref{eq 5.1.14}).\\
		Let $1\leq j\leq r-1$. Then, using \cref{section:quotient:lemma 8}, we have
		\begin{align}
			\tilde{b}_{J_{1},J_{2}}(j)(t^{\prime}y)&=\frac{b(J_{1}(j))(t^{\prime}y)}{b(J_{2}(j))(t^{\prime}y)}
			=\frac{\delta_{J_{1}(j)}(t^{\prime})b(J_{1}(j))(y)}{\delta_{J_{2}(j)}(t^{\prime})b(J_{2}(j))(y)}.\label{eq 5.1.39}
		\end{align}
	From \cref{section:quotient:lemma 7}(1), we have  \begin{align}
		\delta_{J_{1}(j)}(t^{\prime})&=\prod_{l=1}^{r-1}t_{l}^{\langle\delta_{J_{1}(j)},\lambda^{\prime}_{lq}\rangle}
		=\prod_{l=1}^{j}t_{l}\label{eq 5.1.40}
	\end{align}
Similarly, from \cref{section:quotient:lemma 7} (1), we have
\begin{align}
	\delta_{J_{2}(j)}(t^{\prime})=\prod_{l=1}^{j}t_{l}\label{eq 5.1.41}
\end{align}
		Therefore, \cref{eq 5.1.40}, \cref{eq 5.1.41}, we have $\delta_{J_{1}(j)}(t^{\prime})=\delta_{J_{2}(j)}(t^{\prime})$. Hence, 
		from \cref{eq 5.1.39}, we have 	$\tilde{b}_{J_{1},J_{2}}(j)(t^{\prime}y)=\tilde{b}_{J_{1},J_{2}}(j)(y)$.
	\end{proof}
	
	Let $J_{1},J_{2}\in\Pi_{j=1}^{r-1}C_{j}$ and $0\leq j\leq r-1$,
	In view of \cref{section:quotient:lemma 9} and \cref{section:quotient:lemma 10}, 
	$\tilde{b}_{J_{1},J_{2}}(j)$ induces a morphism $b_{J_{i},J_{2}}(j):U(J_{1})\cap U(J_{2})\longrightarrow \mathbb{C}^{\times}$ such that \begin{align}\label{eq 5.1.16}
		b_{J_{1},J_{2}}(j)(\pi_{r-1}(y))=\tilde{b}_{J_{1},J_{2}}(y) \text{\ for all \ $y\in \tilde{U}(J_{1})\cap \tilde{U}(J_{2})$}
	\end{align}
	\begin{lemma}\label{section:quotient:lemma 11 }
		For every $J\in \Pi_{j=1}^{r-1}C_{j}$, we have an isomorphism $$h_{J}:\phi_{r}^{-1}(U(J)))\longrightarrow U(J)\times \mathbb{P}^{r(q-1)}$$ such that the following diagram commutes
		\begin{equation}\label{diagram 2}		
		 \begin{tikzcd}
			\phi_{r}^{-1}(U(J)) \arrow{rr}{h_{J}} \arrow[swap]{dr}{\phi_{r}} & & U(J)\times \mathbb{P}^{r(q-1)}\arrow{dl}{pr_{1}} \\%
			&	U(J)
		\end{tikzcd}
		\end{equation}
	\end{lemma}
	\begin{proof} Let $J=(i_{1},i_{2},\cdots,i_{r-1})$. 
		Let $\{e_{k,r}: k\in C_{r}\}$ be a basis of $\mathbb{C}^{r(q-1)+1}$.		
		
		{\it Step 1}: We construct a morphism $\tilde{h}_{J}:\tilde{V}(J)\longrightarrow U(J)\times\mathbb{P}^{r(q-1)}$.
		
		Let $x=uw_{r}P/P\in \tilde{V}(J)$. Let $u=\prod_{\substack{1\leq j\leq r\\i\in C_{j}}}u_{\beta_{i,j}}(a_{i,j})$, where $a_{i,j}\in\mathbb{C}$.\\
		Since $x\in\tilde{V}(J)$, from \cref{section:quotient:lemma 13}, $\phi_{r}(\pi_{r}(x))\in U(J)$.\\
		Since $\phi=\phi_{r}\circ\pi_{r}$, from \cref{eq 5.1.8}, we have $\phi_{r}(\pi_{r}(x))=\phi(x)=\pi_{r-1}(x[r-1])$. Since $\pi_{r-1}(x[r-1])\in U(J)$, from \cref{section: quotient lemma 6}(1), we have $x[r-1]\in \tilde{U}(J)$. Hence, for all $1\leq p\leq r$, $b(J(p-1))(x[r-1])\neq 0$.\\
		Since $x\in V_{r}$, there exists $(k_{0},r)\in J_{p,r}$ for some $1\leq p\leq r$ such that $a_{k_{0},r}\neq 0$(see \cref{semistability criterion} and \cref{remark 1}). Therefore, $b(J(p-1))(x[r-1])a_{k_{0},r}\neq 0$.
		
		We define $\tilde{h}_{J}:\tilde{V}(J)\longrightarrow U(J)\times\mathbb{P}^{r(q-1)}$ by 
		\begin{align}\label{eq 5.1.26}
			\tilde{h}_{J}(x)=(\phi_{r}(\pi_{r}(x)),[\sum_{p=1}^{r}(\sum_{\substack{(k,r)\in J_{p,r}}}b(J(p-1))(x[r-1])a_{k,r}e_{k,r})]),
		\end{align}
		for all $x=(\prod_{\substack{1\leq j\leq r\\i\in C_{j}}}u_{\beta_{i,j}}(a_{i,j}))w_{r}P/P\in\tilde{V}(J)$.
		
		{\it Step 2}: We prove $\tilde{h}_{J}$ is constant on $T_{J_{r}}$- orbits.\\
		{\it Proof of step 2}: Let $x=uw_{r}P/P\in \tilde{V}(J)$. Let $u=\prod_{\substack{1\leq j\leq r\\i\in C_{j}}}u_{\beta_{i,j}}(a_{i,j})$, where $a_{i,j}\in\mathbb{C}$.\\
		Let $t=\Pi_{j=1}^{r}\lambda_{jq}(t_{j})\in T_{J_{r}}$, where $t_{j}\in\mathbb{G}_{m}$ for $1\leq j\leq r$. We have $tx=tuw_{r}P/P$, where  $tu=\prod_{\substack{1\leq j\leq r\\i\in C_{j}}}u_{\beta_{i,j}}(\beta_{i,j}(t)a_{i,j})$. Since $\pi_{r}(tx)=x$, we have $\phi_{r}(\pi_{r}(tx))=\phi_{r}(\pi_{r}(x))$. Therefore, we have 
		\begin{align}\label{eq 5.1.33}
			\tilde{h}_{J}(tx)=(\phi_{r}(\pi_{r}(x)),[\sum_{p=1}^{r}(\sum_{(k,r)\in J_{p,r}}b(J(p-1))((tx)[r-1])\beta_{k,r}(t)a_{k,r}e_{k,r})])
		\end{align}
		{\it Claim}: For all $1\leq p\leq r$ and $(k,r)\in J_{p,r}$, we have 
		\begin{align*}
			b(J(p-1))((tx)[r-1])\beta_{k,r}(t)=(t_{1}\cdots t_{r})b(J(p-1))(x[r-1]).
		\end{align*}		
		{\it Proof of claim}: Let $t^{\prime}:=\prod_{j=1}^{r-1}\lambda^{\prime}_{jq}(t_{j})$.

		Then from \cref{section:quotient lemma 3}, we have $\beta_{i,j}(t)=\gamma_{i,j}(t^{\prime})$ for all  $1\leq j\leq r-1$ and $i\in C_{j}$.
		Therefore, \begin{align}
			(tu)[r-1]&=\displaystyle\prod_{\substack{1\leq j\leq r-1\\i\in C_{j}}}u^{\prime}_{\gamma_{i,j}}(\beta_{i,j}(t)a_{i,j})\\
			&=\displaystyle\prod_{\substack{1\leq j\leq r-1\\i\in C_{j}}}u^{\prime}_{\gamma_{i,j}}(\gamma_{i,j}(t^{\prime})a_{i,j})\\
			&=t^{\prime}(u[r-1])\label{eq 5.1.27}
		\end{align}
		Since $(tx)[r-1]=(tu)[r-1]w_{r-1}Q/Q$, using \cref{eq 5.1.27}, we have \begin{align}\label{eq 5.1.28}
			(tx)[r-1]=t^{\prime}u[r-1]w_{r-1}Q/Q=t^{\prime}x[r-1]
		\end{align}
		Therefore, for any $2\leq p\leq r$, we have 
		\begin{align}
			b(J(p-1))((tx)[r-1])&=b(J(p-1))(t^{\prime}x[r-1]) (\text{\ using \cref{eq 5.1.28}\ })\\
			&=\delta_{J(p-1)}(t^{\prime})b(J(p-1))(x[r-1])(\text{using \cref{section:quotient:lemma 8}}) \\
			&=(\prod_{k=1}^{p-1}t_{k})b(J(p-1))(x[r-1]) (\text{\ using \cref{section:quotient:lemma 7}(1)})\label{eq 5.1.29}
		\end{align}
			From \cref{section:quotient:lemma 1}(2), for all  $(k,r)\in J_{p,r}$, we have \begin{align}\label{eq 5.1.30}
				\beta_{k,r}(t)=\prod_{k=p}^{r}t_{k}
			\end{align}
			Hence, from \cref{eq 5.1.29}, and \cref{eq 5.1.30}, for $2\leq p\leq r$ and $(k,r)\in J_{p,r}$, we have 
			\begin{align}\label{eq 5.1.31}
				b(J(p-1))((tx)[r-1])\beta_{k,r}(t)=(t_{1}\cdots t_{r})b(J(p-1))(x[r-1])
			\end{align}
			Let $p=1$. For any $(k,r)\in J_{1,r}$, we have $\beta_{k,r}(t)=\Pi_{j=1}^{r}t_{j}$ (see \cref{section:quotient:lemma 1}(2)).
			Therefore, using  \cref{eq 5.1.14}, we have  \begin{align}\label{eq 5.1.32}
				b(J(0))((tx)[r-1])\beta_{k,r}(t)=(t_{1}\cdots t_{r})b(J(0))(x[r-1])
			\end{align}
			The claim follows from \cref{eq 5.1.31} and \cref{eq 5.1.32}.
			
			Now using above claim, from \cref{eq 5.1.33}, we have
			\begin{align*}
				\tilde{h}_{J}(tx)&=(\phi_{r}(\pi_{r}(x)),[(t_{1}\cdots t_{r})\sum_{p=1}^{r}(\sum_{\substack{(k,r)\in J_{p,r}}}b(J(p-1))(x[r-1])a_{k,r}e_{k,r})])\\
				&=(\phi_{r}(\pi_{r}(x)),[\sum_{p=1}^{r}(\sum_{\substack{(k,r)\in J_{p,r}}}b(J(p-1))(x[r-1])a_{k,r}e_{k,r})])\\
				&=\tilde{h}_{J}(x) (\text{\ using \cref{eq 5.1.26}\ })
			\end{align*}
			This proves that $\tilde{h}_{J}$ is constant on $T_{J_{r}}$-orbits.\\
			Hence, $\tilde{h}_{J}$
			induces a morphism  $h_{J}:\pi_{r}(\tilde{V}(J))\longrightarrow U(J)\times \mathbb{P}^{r(q-1)}$ such that $h_{J}(\pi_{r}(x))=\tilde{h}_{J}(x)$ for all $x\in \tilde{V}(J)$.
			Recall that $V(J)=\pi_{r}(\tilde{V}(J))$. Further, in view of \cref{section:quotient:lemma 13}, we have $\phi_{r}^{-1}(U(J))=V(J)$. \\
			Thus, we have $h_{J}:\phi_{r}^{-1}(U(J))\longrightarrow U(J)\times \mathbb{P}^{r(q-1)}$ such that $h_{J}\circ\pi_{r}=\tilde{h}_{J}$.
			
			Let $x\in  \tilde{V}(J)$. Then we have $pr_{1}(h_{J}(\pi_{r}(x)))=pr_{1}(\tilde{h}_{J}(x))=\phi_{r}(\pi_{r}(x))$. Hence, $\phi_{r}=pr_{1}\circ h_{J}$.
			Therefore, the diagram in \cref{diagram 2} commutes.
			
			{\it $h_{J}$ is injective}: Let $x_{1},x_{2}\in \tilde{V}(J)$. Let $x_{1}=u_{1}w_{r}P/P$, $x_{2}=u_{2}w_{r}P/P$ where $u_{1}=\Pi_{\substack{1\leq j\leq r\\i\in C_{j}}}u_{\beta_{i,j}}(a_{i,j})$ and $u_{2}=\Pi_{\substack{1\leq j\leq r\\i\in C_{j}}}u_{\beta_{i,j}}(b_{i,j})$, $a_{i,j},b_{i,j}\in\mathbb{C}$.
			
			Assume that $h_{J}(\pi_{r}(x_{1}))=h_{J}(\pi_{r}(x_{2}))$. Since $\tilde{h}_{J}=h_{J}\circ\pi_r$, from \cref{eq 5.1.26}, we have \begin{align}\label{eq 5.1.43}
				\phi_{r}(\pi_{r}(x_{1}))=\phi_{r}(\pi_{r}(x_{2})),
			\end{align}
		and
			\begin{align}\label{eq 5.1.19}
				[\sum_{p=1}^{r}(\sum_{\substack{(k,r)\in J_{p,r}}}b(J(p-1))(x_{1}[r-1])a_{k,r}e_{k,r})]=[\sum_{p=1}^{r}(\sum_{\substack{(k,r)\in J_{p,r}}}b(J(p-1))(x_{2}[r-1])b_{k,r}e_{k,r})]
			\end{align}
			Since $\phi=\phi_{r}\circ\pi_r$, from \cref{eq 5.1.8} and \cref{eq 5.1.43} we have $\pi_{r-1}(x_{1}[r-1])=\pi_{r-1}(x_{2}[r-1])$. Hence, there exists $t^{\prime}\in T^{\prime}_{J_{r-1}}$ such that $x_{2}[r-1]=t^{\prime}x_{1}[r-1]$.\\
			Let $t^{\prime}=\prod_{j=1}^{r-1}\lambda^{\prime}_{jq}(t_{j})$, where $t_{j}\in\mathbb{G}_{m}$ for $1\leq j\leq r-1$. 
			
			From \cref{eq 5.1.18}, we have $u_{2}[r-1]w_{r-1}Q/Q=t^{\prime}u_{1}[r-1]w_{r-1}Q/Q$. Now from
			\cref{eq 5.1.38}, we have 
			\begin{align*}
				(\prod_{\substack{1\leq j\leq r-1\\i\in C_{j}}}u^{\prime}_{\gamma_{i,j}}(b_{i,j}))w_{r-1}Q/Q&=t^{\prime}(\prod_{\substack{1\leq j\leq r-1\\i\in C_{j}}}u^{\prime}_{\gamma_{i,j}}(a_{i,j}))w_{r-1}Q/Q\\&=\prod_{\substack{1\leq j\leq r-1\\i\in C_{j}}}u^{\prime}_{\gamma_{i,j}}(\gamma_{i,j}(t^{\prime})a_{i,j})w_{r-1}Q/Q
			\end{align*}

			Hence, for every $1\leq j\leq r-1$ and $i\in C_{j}$, we have 
			\begin{align}\label{eq 5.1.23}
				b_{i,j}=\gamma_{i,j}(t^{\prime})a_{i,j}
			\end{align}
			From \cref{eq 5.1.19}, there exists  $c\in\mathbb{C}^{\times}$ such that, for all $1\leq p\leq r$ and for all $(k,r)\in J_{p,r}$ we have \begin{align}\label{eq 5.1.21}
				b(J(p-1))(x_{2}[r-1])b_{k,r}=c\cdot b(J(p-1))(x_{1}[r-1])a_{k,r}
			\end{align}
			Define $t_{r}:=c(\Pi_{j=1}^{r-1}t_{j})^{-1}$. Then, we have $c=\prod_{j=1}^{r}t_{j}$. Let $t=\prod_{l=1}^{r}\lambda_{lq}(t_{l})$. Then we have $t\in T_{J_{r}}$.
			
			Let $p=1$. Then using \cref{eq 5.1.14} and \cref{eq 5.1.21}, for all $(k,r)\in J_{1,r}$, we have \begin{align}\label{eq 5.1.20}
				b_{k,r}=ca_{k,r}=
				(t_{1}\cdots t_{r})a_{k,r}
			\end{align} 
			From \cref{section:quotient:lemma 1} $(2)$, for all $(k,r)\in J_{1,r}$, we have $\beta_{k,r}(t)=\Pi_{j=1}^{r}t_{j}$. Therefore, from \cref{eq 5.1.20}, for all $(k,r)\in J_{1,r}$, we have \begin{align}\label{eq 5.1.42}
				b_{k,r}=\beta_{k,r}(t)a_{k,r}
			\end{align}
			Let $2\leq p\leq r$. Since $x_{2}[r-1]=t^{\prime}x_{1}[r-1]$, from \cref{section:quotient:lemma 8}, we have \begin{align*}
				b(J(p-1))(x_{2}[r-1])&=	b(J(p-1))(t^{\prime}x_{1}[r-1])\\
				&=\delta_{J(p-1)}(t^{\prime})b(J(p-1))(x_{1}[r-1])
			\end{align*} 
			Therefore, from \cref{eq 5.1.21}, we have  $b_{k,r}=c(\delta_{J(p-1)}(t^{\prime}))^{-1}a_{k,r}$. From \cref{section:quotient:lemma 7}(1), $\delta_{J(p-1)}(t^{\prime})=(\Pi_{j=1}^{p-1}t_{j})$. Therefore, $c(\delta_{J(p-1)}(t^{\prime}))^{-1}=\Pi_{j=p}^{r}t_{j}$. Hence,
			\begin{align}\label{eq 5.1.22}
				b_{k,r}=(\Pi_{j=p}^{r}t_{j})a_{k,r}
			\end{align}

			Therefore, using \cref{section:quotient:lemma 1}(2), for all $(k,r)\in J_{p,r}$,  we have $\beta_{k,r}(t)=\Pi_{j=p}^{r}t_{j}$. Hence, for $2\leq p\leq r$ and $(k,r)\in J_{p,r}$, from \cref{eq 5.1.22},we have 
			\begin{align}\label{eq 5.1.24}
				b_{k,r}=\beta_{k,r}(t)a_{k,r}
			\end{align}
			From  \cref{eq 5.1.23} and \cref{section:quotient lemma 3}, for $1\leq j\leq r-1$ and $i\in C_{j}$, we have \begin{align}\label{eq 5.1.25}
				b_{i,j}=\beta_{i,j}(t)a_{i,j}
			\end{align}
			Thus, from \cref{eq 5.1.42}, \cref{eq 5.1.24} and \cref{eq 5.1.25}, for all $1\leq j\leq r$ and $i\in C_{j}$, we have $b_{i,j}=\beta_{i,j}(t)a_{i,j}$. Hence, we have
			$x_{2}=tx_{1}$ for $t\in T_{J_{r}}$. Therefore, $\pi_{r}(x_{1})=\pi_{r}(x_{2})$.
			
			{\it $h_{J}$ is surjective}: Let $(\pi_{r-1}(y),[\sum_{p=1}^{r}(\sum_{(k,r)\in J_{p,r}}c_{k,r}e_{k,r})])\in U(J)\times \mathbb{P}^{r(q-1)}$, where $y=u^{\prime}w_{r-1}Q/Q\in\tilde{U}(J)$, $u^{\prime}=\prod_{\substack{1\leq j\leq r-1\\i\in C_{j}}}u^{\prime}_{\gamma_{i,j}}(a_{i,j})$ and $a_{i,j},c_{k,r}\in\mathbb{C}$, for all $1\leq j\leq r-1$, $i\in C_{j}$, $k\in C_{r}$.
			
			Let \begin{align}\label{eq 5.1.34}
				u=(\displaystyle\prod_{\substack{1\leq j\leq r-1\\i\in C_{j}}}u_{\beta_{i,j}}(a_{i,j}))(\prod_{p=1}^{r}(\prod_{(k,r)\in J_{p,r}}u_{\beta_{k,r}}(\frac{c_{k,r}}{b(J(p-1))(y)})))
			\end{align}
			
			Let $x=uw_{r}P/P$. 
			We have $c_{k_{0},r}\neq 0$ for some $(k_{0},r)\in J_{p,r}$. Since $b(J(p-1))(y)\neq 0$, we have $\frac{c_{k_{0},r}}{b(J(p-1))(y)}\neq 0$. Since $y\in \tilde{U}(J)$, we have  $a_{i_{j},j}\neq 0$ for all $1\leq j\leq r-1$. Hence, using \cref{semistability criterion} $x\in V_{r}$. Since $X_{\beta_{i_{j},j}}(x)=a_{i_{j},j}\neq 0$ for all $1\leq j\leq r-1$, from \cref{eq 5.1.10}, we have $x\in \tilde{V}(J)$.\\
			Since $x[r-1]=y$, we have $\phi_{r}(\pi_{r}(x))=\pi_{r-1}(x[r-1])=\pi_{r-1}(y)$. Therefore, from \cref{eq 5.1.26}, we have
			\begin{align*}
				\tilde{h}_{J}(x)&=(\pi_{r-1}(y),[\sum_{p=1}^{r}(\sum_{(k,r)\in J_{p,r}}b(J(p-1))(y)\frac{c_{k,r}}{b(J(p-1))(y)}e_{k,r})])\\
				&=(\pi_{r-1}(y),[\sum_{p=1}^{r}(\sum_{(k,r)\in J_{p,r}}c_{k,r}e_{k,r})])
			\end{align*}
			Thus, $\tilde{h}_{J}$ is surjective.  Hence $h_{J}$ is surjective. Therefore, $h_{J}$ is bijective. We can see that $h_{J}$ is birational.
			Since $Y_{r-1}$ is normal and $U(J)$ is an open subset of $Y_{r-1}$, $U(J)$ is normal as well. Hence, $U(J)\times\mathbb{P}^{r(q-1)}$ is normal. Therefore, using Zariski's main theorem, we have $h_{J}$ is an isomorphism (see \cite[Theorem 5.2.8]{springer}).
		\end{proof}		
		\begin{lemma} \label{section:quotient:lemma 12}
			Let $J_{1},J_{2}\in \Pi_{j=1}^{r-1}C_{j}$. Then the transition automorphism
			$$h_{J_{1}}\circ h_{J_{2}}^{-1}:(U(J_{1})\cap U(J_{2}))\times \mathbb{P}^{r(q-1)}\longrightarrow (U(J_{1})\cap U(J_{2}))\times \mathbb{P}^{r(q-1)}$$ is given by 	 \begin{align*}
				h_{J_{1}}\circ h_{J_{2}}^{-1}(&\pi_{r-1}(y),[\sum_{p=1}^{r}(\sum_{(k,r)\in J_{p,r}}x_{k,r}e_{k,r})])=(\pi_{r-1}(y),[\sum_{p=1}^{r}(\sum_{(k,r)\in J_{p,r}}y_{k,r}e_{k,r})]),\\&\text{for all $y\in \tilde{U}(J_{1})\cap \tilde{U}(J_{2})$ and $[\sum_{p=1}^{r}(\sum_{(k,r)\in J_{p,r}}x_{k,r}e_{k,r})]\in\mathbb{P}^{r(q-1)}$}
			\end{align*}
			where, for all $1\leq p\leq r$ and $(k,r)\in J_{p,r}$,
			$y_{k,r}=b_{J_{1},J_{2}}(p-1)(\pi_{r-1}(y))x_{k,r}$.
		\end{lemma}
		\begin{proof}Let $y=u^{\prime}w_{r-1}Q/Q$, where $u^{\prime}\in \tilde{U}(J_{1})\cap\tilde{U}(J_{2})$. Let $u^{\prime}=\prod_{\substack{1\leq j\leq r-1\\i\in C_{j}}}u^{\prime}_{\gamma_{i,j}}(a_{i,j})$, where $a_{i,j}\in\mathbb{C}$. Then from the expression in \cref{eq 5.1.34}, we have \begin{align*}
				h_{J_{2}}^{-1}(\pi_{r-1}(y),[\sum_{p=1}^{r}(\sum_{(k,r)\in J_{p,r}}x_{k,r}e_{k,r})])=\pi_{r}(uw_{r}P/P),
			\end{align*}
			where 
			$u=(\prod_{\substack{1\leq j\leq r-1\\i\in C_{j}}}u_{\beta_{i,j}}(a_{i,j}))(\prod_{p=1}^{r}(\prod_{(k,r)\in J_{p,r}}u_{\beta_{k,r}}(\frac{x_{k,r}}{b(J_{2}(p-1))(y)})))$.
			
			Let $x=uw_{r}P/P$. Then, we have $x[r-1]=y$. Hence $\phi_{r}(\pi_{r}(x))=\pi_{r-1}(x[r-1])=\pi_{r-1}(y)$. Since $h_{J}\circ\pi_{r}=\tilde{h}_{J}$, from \cref{eq 5.1.26}), we have 
			\begin{align}\label{eq 5.1.17}
				h_{J_{1}}(\pi_{r}(x))&=(\pi_{r-1}(y),[\sum_{p=1}^{r}(\sum_{(k,r)\in J_{p,r}}\frac{b(J_{1}(p-1))(y)}{b(J_{2}(p-1))(y)}x_{k,r}e_{k,r})])
			\end{align}
			From \cref{eq 5.1.15} and \cref{eq 5.1.16}), we have $\frac{b(J_{1}(p-1))(y)}{b(J_{2}(p-1))(y)}= b_{J_{1},J_{2}}(p-1)(\pi_{r-1}(y))$. Hence the proof follows from \cref{eq 5.1.17}.
		\end{proof}
		\begin{theorem}\label{main result 1}
			$Y_{r}$ is a $\mathbb{P}^{r(q-1)}$-bundle over $Y_{r-1}$.
		\end{theorem}
		\begin{proof}From \cref{section 5.1}, we have an open covering $Y_{r-1}=\cup_{J\in\Pi_{j=1}^{r-1}C_{j}}U(J)$ and a morphism $\phi_{r}:Y_{r}\longrightarrow Y_{r-1}$. From \cref{section:quotient:lemma 11 }, we have $Y_{r}$ is locally isomorphic to $U(J)\times\mathbb{P}^{r(q-1)}$. Let $J_{1},J_{2}\in\Pi_{j=1}^{r-1}C_{j}$. Then from \cref{section:quotient:lemma 12}, for any affine open subset $Z\subseteq U(J_{1})\cap U(J_{2})$ we have, the transition automorphism $h_{J_{1}}\circ h_{J_{2}}^{-1}:Z\times \mathbb{P}^{r(q-1)}\longrightarrow Z\times\mathbb{P}^{r(q-1)}$ is given by  
			\begin{align*}
				h_{J_{1}}\circ h_{J_{2}}^{-1}(z,[\sum_{p=1}^{r}(\sum_{(k,r)\in J_{p,r}}x_{k,r}e_{k,r})])=(z,[\sum_{p=1}^{r}(\sum_{(k,r)\in J_{p,r}}y_{k,r}e_{k,r})]), \text{ for all $z\in Z$}
			\end{align*}
		where $y_{k,r}=b_{J_{1},J_{2}}(p-1)\!\restriction_Z(z)x_{k,r}$ for all $(k,r)\in J_{p,r}$. Since for all $1\leq p\leq r$, $b_{J_{1},J_{2}}(p-1)\in \mathcal{O}_{U(J_{1})\cap U(J_{2})}(U(J_{1})\cap U(J_{2}))$, we have  $b_{J_{1},J_{2}}(p-1)\!\restriction_Z\in\mathbb{C}[Z]$. Therefore, the transition automorphis on $Z\times\mathbb{P}^{r(q-1)}$ are given by $\mathbb{C}[Z]$-linear automorphism of $(\mathbb{C}[Z])[X_{0},X_{1},\ldots,X_{r(q-1)}]$. Now the proof follows from \cite[ch.2, exercise 7.10]{Ha}.	
		\end{proof} 
		\section{Description of $Y_{r}$ as projective bundle over $Y_{r-1}$}
		In this section, we prove that there exists line bundles $\mathcal{L}_{0},\mathcal{L}_{1},\ldots,\mathcal{L}_{r-1}$ on $Y_{r-1}$ such that $Y_{r}$ is isomorphic to the projective space bundle $\mathbb{P}(\mathcal{L}_{0}^{\oplus q}\oplus(\bigoplus_{j=1}^{r-1} \mathcal{L}_{j}^{\oplus q-1}))$ over $Y_{r-1}$. We follow the same notations as in \cref{section 5}.
	\begin{lemma}\label{section:projective bundle:lemma 1}
		For any three $J_{1},J_{2},J_{3}\in  \Pi_{k=1}^{r-1}C_{k}$, and any $0\leq j\leq r-1$,  we have 
		\begin{align*}
			b_{J_{1},J_{2}}(j)^{-1}(z)&=b_{J_{2},J_{1}}(j)(z)\ \text{for all $z\in U(J_{1})\cap U(J_{2})$}\\
			b_{J_{1},J_{2}}(j)(z)\cdot b_{J_{2},J_{3}}(j)(z)&=b_{J_{1},J_{3}}(j)(z)\ \text{for all $z\in U(J_{1})\cap U(J_{2})\cap U(J_{3})$}\\
			b_{J_{1},J_{1}}(j)(z)&=1\ \text{for all $z\in U(J_{1})$}
		\end{align*}
	\end{lemma}
	\begin{proof} 
		Let $z\in U(J_{1})\cap U(J_{2})$. From \cref{section:quotient:lemma 10}(1), $z=\pi_{r-1}(y)$, for some  $y\in \tilde{U}(J_{1})\cap \tilde{U}(J_{2})$. Let $0\leq j\leq r-1$. From \cref{eq 5.1.15} and \cref{eq 5.1.16}, we have
		\begin{align*}
			b_{J_{2},J_{1}}(j)(z)=\frac{b(J_{2}(j))(y)}{b(J_{1}(j))(y)}
			=(\frac{b(J_{1}(j))(y)}{b(J_{2}(j))(y)})^{-1}
			&=(\tilde{b}_{J_{1},J_{2}}(j)(y))^{-1} (\text{\ using \cref{eq 5.1.15}})\\
			&=(b_{J_{1},J_{2}}(j)(\pi_{r-1}(y)))^{-1} (\text{\ using \cref{eq 5.1.16}}).
		\end{align*}
		Let $z\in U(J_{1})\cap U(J_{2})\cap U(J_{3})$. 
		From \cref{section:quotient:lemma 10}(2), we have $z=\pi_{r-1}(y)$ for some $y\in \tilde{U}(J_{1})\cap \tilde{U}(J_{2})\cap \tilde{U}(J_{3})$. From \cref{eq 5.1.15} and \cref{eq 5.1.16}, we have \begin{align*}
			b_{J_{1},J_{2}}(j)(z)b_{J_{2},J_{3}}(j)(z)=\frac{b(J_{1}(j))(y)}{b(J_{2}(j))(y)}\frac{b(J_{2}(j))(y)}{b(J_{3}(j))(y)}
			&=\frac{b(J_{1}(j))(y)}{b(J_{3}(j))(y)}\\
			&=\tilde{b}_{J_{1},J_{3}}(j)(y) \text{(see \cref{eq 5.1.15})}\\
			&=b_{J_{1},J_{3}}(j)(z) \text{(see \cref{eq 5.1.16})}.
		\end{align*}
	Let $z=\pi_{r-1}(y)\in U(J_{1})$ for some $y\in \tilde{U}(J_{1})$. Then from  \cref{eq 5.1.16} and \cref{eq 5.1.15}, we have $b_{J_{1},J_{1}}(j)(z)=\frac{b(J_{1}(j))(y)}{b(J_{1}(j))(y)}=1$.
	\end{proof}
	Fix $0\leq j\leq r-1$. Using \cref{section:projective bundle:lemma 1}, from \cite[section 1.3]{LePotier} we can see
	$$\{b_{J_{1},J_{2}}(j):U(J_{1})\cap U(J_{2})\longrightarrow\mathbb{C}^{\times}|J_{1},J_{2}\in\Pi_{j=1}^{r-1}C_{j}\}$$ gives transition functions for a line bundle on $Y_{r-1}$. We denote the corresponding line bundle by $\pi_{j}:\mathcal{L}_{j}\longrightarrow Y_{r-1}$.
	Note that $\mathcal{L}_{0}$ is the trivial line bundle on $Y_{r-1}$.\\
	Let $\mathcal{E}:=\mathcal{L}_{0}^{\oplus q}\oplus(\bigoplus_{j=1}^{r-1} \mathcal{L}_{j}^{\oplus q-1})$. Then $\mathcal{E}$ is a vector bundle of rank $r(q-1)+1$ on $Y_{r-1}$. Let $\pi:\mathcal{E}\longrightarrow Y_{r-1}$ denote the associated morphism. \\
	For any $J_{1},J_{2}\in\Pi_{j=1}^{r-1}C_{j}$, the transition functions for $\mathcal{E}$ is given by $g_{J_{1},J_{2}}:U(J_{1})\cap U(J_{2})\longrightarrow GL(r(q-1)+1,\mathbb{C})$, where for all  $z\in U(J_{1})\cap U(J_{2})$,
	$g_{J_{1},J_{2}}(z)$ is the block diagonal matrix  
	\begin{align}\label{eq 6.1.10}
		g_{J_{1},J_{2}}(z)&=
\begin{pmatrix}
	A_1 & 0   & \cdots & 0 \\
	0   & A_2 & \cdots & 0 \\
	\vdots & \vdots & \ddots & \vdots \\
	0   & 0   & \cdots & A_r
\end{pmatrix}\end{align}
	where \begin{align}
		A_{1}&=diag(\underbrace{b_{J_{1},J_{2}}(0)(z),\ldots,b_{J_{1},J_{2}}(0)(z)}_{\text{$q$ times}})\label{eq 6.1.11}\\
		\text{\ for $2\leq j\leq r$}, A_{j}&=diag(\underbrace{b_{J_{1},J_{2}}(j-1)(z),\ldots,b_{J_{1},J_{2}}(j-1)(z)}_{\text{$q-1$ times}})\label{eq 6.1.12}
	\end{align}

	Recall the definition of $\mathbb{P}(\mathcal{E})$ from \cref{definition:projective bundle},
	\begin{align}\label{eq:definiton P(E)}
		\mathbb{P}(\mathcal{E}):=\frac{\bigsqcup_{J\in \Pi_{j=1}^{r-1}C_{j}}(U(J)\times\mathbb{P}^{r(q-1)})}{\sim},
	\end{align} 
	where $(z_{1},[v])\sim(z_{2},[w])$ if and only if $z_{1}=z_{2}$, $[g_{J_{2},J_{1}}(z_{1})(v)]=[w]$ for $z_{1}\in U(J_{1})$, $z_{2}\in U(J_{2})$ and $[v],[w]\in\mathbb{P}^{r(q-1)}$. Since $b_{J_{1},J_{2}}(j)$ satisfies cocycle conditios for all $0\leq j\leq r-1$, $g_{J_{1},J_{2}}$ also satisfies cocycle conditions.\\ We denote the equivalence class of $(z,[v])$ by $[(z,[v])]$. We have a morphism $\bar{\pi}:\mathbb{P}(\mathcal{E})\longrightarrow Y_{r-1}$ given by $\bar{\pi}([(z,[v])])=z$.
	
	Let $\{e_{k,r}:k\in C_{r}\}$ be a basis of $\mathbb{C}^{r(q-1)+1}$.
	
	\begin{theorem}\label{section:projective bundle:th 1} We have an isomorphism 
		$\Psi:Y_{r}\longrightarrow\mathbb{P}(\mathcal{E})$ such that the following diagram commutes
	
		\begin{equation}\label{eq:diagram 2}
				 \begin{tikzcd}
			Y_{r}\arrow{rr}{\Psi} \arrow[swap]{dr}{\phi_{r}} & & \mathbb{P}(\mathcal{E})\arrow{dl}{\bar{\pi}} \\%
			&	Y_{r-1}
		\end{tikzcd}
	\end{equation}
	\end{theorem}
	\begin{proof}
		We have $Y_{r}=\displaystyle\cup_{J\in\Pi_{j=1}^{r-1}C_{j}}V(J)$. Define $\Psi_{J}:V(J)\longrightarrow\mathbb{P}(\mathcal{E})$ by $\Psi_{J}(y):=[h_{_{J}}(y)]$ for all $y\in V(J)$. Here $[h_{J}(y)]$ is the equivalence class of $h_{J}(y)$ under the equivalence relation $\sim$ in \cref{eq:definiton P(E)}.\\
		{\it Claim}: Let $J_{1},J_{2}\in \Pi_{j=1}^{r-1}C_{j}$. Then for any $y\in V(J_{1})\cap V(J_{2})$, we have $\Psi_{J_{1}}(y)=\Psi_{J_{2}}(y)$.\\
		{\it Proof of claim}: Let $y\in V(J_{1})\cap V(J_{2})$. From \cref{section:quotient:lemma 14}, we have $y=\pi_{r}(x)$ for some $x=(\prod_{\substack{1\leq j\leq r\\i\in C_{j}}}u_{\beta_{i,j}}(a_{i,j}))w_{r}P/P\in \tilde{V}(J_{1})\cap \tilde{V}(J_{2})$. Since $h_{J}=\tilde{h}_{J}\circ\pi_{r}$, using \cref{eq 5.1.26}, we have
		$h_{J_{1}}(y)=(\phi_{r}(y),[v])$ and $h_{J_{2}}(y)=(\phi_{r}(y),[w])$, where 
		\begin{align}\label{eq 6.1.8}
			[v]=[\sum_{p=1}^{r}(\sum_{(k,r)\in J_{p,r}}b(J_{1}(p-1))(x[r-1])a_{k,r}e_{k,r})]
		\end{align}
		\begin{align}\label{eq 6.1.7}
			[w]=[\sum_{p=1}^{r}(\sum_{(k,r)\in J_{p,r}}b(J_{2}(p-1))(x[r-1])a_{k,r}e_{k,r})]
		\end{align}
		From \cref{eq 6.1.7}, we have
		\begin{align}\label{eq 6.1.13}
			[w]&=[\sum_{p=1}^{r}(\sum_{(k,r)\in J_{p,r}}\frac{b(J_{2}(p-1))(x[r-1])}{b(J_{1}(p-1))(x[r-1])}b(J_{1}(p-1))(x[r-1])a_{k,r}e_{k,r})]
		\end{align}
Now using \cref{eq 5.1.15} and \cref{eq 5.1.16}, we have $\frac{b(J_{2}(p-1))(x[r-1])}{b(J_{1}(p-1))(x[r-1])}=b_{J_{2},J_{1}}(p-1)(\pi_{r-1}(x[r-1]))$.
Therefore, from \cref{eq 6.1.13} we have \begin{align}
			[w]&=[\sum_{p=1}^{r}(\sum_{(k,r)\in J_{p,r}}b_{J_{2},J_{1}}(p-1)(\pi_{r-1}(x[r-1]))b(J_{1}(p-1))(x[r-1])a_{k,r}e_{k,r})]\label{eq 6.1.9}	
		\end{align}
		Since  $\phi_{r}(y)=\pi_{r-1}(x[r-1])$, from \cref{eq 6.1.9}, we have
		\begin{align}
			[w]&=[\sum_{p=1}^{r}(\sum_{(k,r)\in J_{p,r}}b_{J_{2},J_{1}}(p-1)(\phi_{r}(y))b(J_{1}(p-1))(x[r-1])a_{k,r}e_{k,r})]\label{eq:section 6:1.4}\end{align}
	On the other hand, from \cref{eq 6.1.10} and \cref{eq 6.1.8} , we have 
		\begin{align}
			[g_{J_{2},J_{1}}(\phi_{r}(y))(v)]&=[\sum_{p=1}^{r}(\sum_{(k,r)\in J_{p,r}}b_{J_{2},J_{1}}(p-1)(\phi_{r}(y))b(J_{1}(p-1))(x[r-1])a_{k,r}e_{k,r})]\label{eq:section 6:1.5}
		\end{align}	
Hence, from \cref{eq:section 6:1.4} and \cref{eq:section 6:1.5}, we have $[w]=[g_{J_{2},J_{1}}(\phi_{r}(y))(v)]$.
		Therefore, we have $[h_{J_{1}}(y)]=[h_{J_{2}}(y)]$. Hence $\Psi_{J_{1}}(y)=\Psi_{J_{2}}(y)$. Thus the morphisms $\Psi_{J}$'s glue together to give a well defined morphism $\Psi:Y_{r}\longrightarrow\mathbb{P}(\mathcal{E})$. 
		
		Further, for $y\in V(J)$, $\bar{\pi}(\Psi(y))=\bar{\pi}([h_{J}(y)])=\phi_{r}(y)$. Hence the diagram in \cref{eq:diagram 2} is commutative. 
		
		{\it $\Psi$ is injective}: Let $y_{1},y_{2}\in Y_{r}$ and $\Psi(y_{1})=\Psi(y_{2})$. Let $y_{1}\in V(J_{1})$ and $y_{2}\in V(J_{2})$. Let $y_{1}=\pi_{r}(x_{1})$ and $y_{2}=\pi_{r}(x_{2})$, for some $x_{1}\in \tilde{V}(J_{1})$, $x_{2}\in \tilde{V}(J_{2})$ (see \cref{eq 5.1.44}). Let
		\begin{align*}
			x_{1}&=u_{1}w_{r}P/P, u_{1}=\prod_{\substack{1\leq j\leq r\\i\in C_{j}}}u_{\beta_{i,j}}(a_{i,j}), a_{i,j}\in\mathbb{C}\\
			x_{2}&=u_{2}w_{r}P/P, u_{2}=\prod_{\substack{1\leq j\leq r\\i\in C_{j}}}u_{\beta_{i,j}}(b_{i,j}), b_{i,j}\in\mathbb{C}
		\end{align*}
		Then we have \begin{align}
			h_{J_{1}}(y_{1})&=(\phi_{r}(y_{1}),[v])\label{eq 6.1.4}\\
			h_{J_{2}}(y_{2})&=(\phi_{r}(y_{2}),[w])\label{eq 6.1.5}
		\end{align}
		where 
		\begin{align}
			[v]&=[\sum_{p=1}^{r}(\sum_{\substack{(k,r)\in J_{p,r}}}b(J_{1}(p-1))(x_{1}[r-1])a_{k,r}e_{k,r})]\label{eq:section 6:1.1}\\
			[w]&=[\sum_{p=1}^{r}(\sum_{\substack{(k,r)\in J_{p,r}}}b(J_{2}(p-1))(x_{2}[r-1])b_{k,r}e_{k,r})]\label{eq:section 6:1.2}
		\end{align}
		Sicne $\Psi(y_{1})=\Psi(y_{2})$, we have $[h_{J_{1}}(y_{1})]=[h_{J_{2}}(y_{2})]$. Hence, we have 
		\begin{align}\label{eq 6.1.1}
			\phi_{r}(y_{1})=\phi_{r}(y_{2})
		\end{align}
		\begin{align}\label{eq 6.1.2}
			[w]=[g_{J_{2},J_{1}}(\phi_{r}(y_{1}))(v)]
		\end{align}
		Using the definition of $g_{J_{2},J_{1}}(\phi_{r}(y_{1}))$ in \cref{eq 6.1.10} (also see \cref{eq 6.1.11} and \cref{eq 6.1.12}) together with the description of $[v]$ in \cref{eq:section 6:1.1}, we have 	
		\begin{align}\label{eq:section 6:1.3}
			[g_{J_{2},J_{1}}(\phi_{r}(y_{1}))(v)]&=[\sum_{p=1}^{r}(\sum_{\substack{(k,r)\in J_{p,r}}}\frac{b(J_{2}(p-1))(x_{1}[r-1])}{b(J_{1}(p-1))(x_{1}[r-1])}b(J_{1}(p-1))(x_{1}[r-1])a_{k,r}e_{k,r})]\end{align}
	Therefore, from \cref{eq 6.1.2} and \cref{eq:section 6:1.3} we have 
		\begin{align}
			[w]=[\sum_{p=1}^{r}(\sum_{\substack{(k,r)\in J_{p,r}}}b(J_{2}(p-1))(x_{1}[r-1])a_{k,r}e_{k,r})]\label{eq 6.1.3}
		\end{align}
		Since $\phi_{r}(y_{1})=\phi_{r}(y_{2})\in U(J_{2})$, from \cref{section:quotient:lemma 13}, we have $y_{1},y_{2}\in V(J_{2})$. Since $y_{1}=\pi_{r}(x_{1})$ and $h_{J_{2}}\circ\pi_{r}=\tilde{h}_{J_{2}}$, we have  $h_{J_{2}}(y_{1})=\tilde{h}_{J_{2}}(x_{1})$. Therefore, from \cref{eq 5.1.26} we have
		\begin{align}
			h_{J_{2}}(y_{1})&=(\phi_{r}(y_{1}),[\sum_{p=1}^{r}(\sum_{\substack{(k,r)\in J_{p,r}}}b(J_{2}(p-1))(x_{1}[r-1])a_{k,r}e_{k,r})])\\
			&=(\phi_{r}(y_{1}),[w]) (\text{\ using \cref{eq 6.1.3}\ })\label{eq 6.1.6}
		\end{align}
		Using \cref{eq 6.1.1}, from \cref{eq 6.1.5} and \cref{eq 6.1.6}, we have
		$h_{J_{2}}(y_{1})=h_{J_{2}}(y_{2})$. Since $h_{J_{2}}$ is injective, we have $y_{1}=y_{2}$.\\
		Since $h_{J}$ is surjective, $\Psi_{J}$ is surjective. Hence $\Psi$ is surjective. Therefore, $\Psi$ is bijective. It is easy to see that $\Psi(V(J))$ is isomorphic to the open subset $\bar{\pi}(U(J)\times \mathbb{P}^{r(q-1)})$. Hence $\Psi$ is birational. Since $Y_{r-1}$ is normal, $\mathbb{P}(\mathcal{E})$ is normal. Now using Zaraski's main theorem, we conclude that $\Psi$ is an isomorphism(see \cite[Theorem 5.2.8]{springer}).
	\end{proof}
\subsection{$Y_{r}$ is stage $r$ of a generalized Bott tower} \label{section 6.1}Let $r,q\in\mathbb{N}$ be such that $r\ge2$ and $q\ge 2$.
	For $1\leq k\leq r$, let $n_{k}=kq+1$ and $G_{k}=PSL(n_{k},\mathbb{C})$. Let $T_{k}$ be a maximal torus of $G_{k}$ and  $B_{k}$ be a Borel subgroup of $G_{k}$ containing $T_{k}$. Let $S_{k}:=\{\alpha_{1(k)},\alpha_{2(k)},\cdots,\alpha_{n_{k}-1(k)}\}$ denote the set of simple roots of $G_{k}$ with respect to $(B_{k},T_{k})$. Let $\{\omega_{1(k)},\omega_{2(k)},\cdots,\omega_{n_{k}-1(k)}\}$ denote the fundamental weights associated to $S_{k}$. 
	
	Let $\{\lambda_{1(k)},\lambda_{2(k)},\cdots,\lambda_{n_{k}-1(k)}\}$ denote the fundamental one parameter subgroups of $T_{k}$ dual to $S_{k}$. Let $Q_{k}$ denote the maximal parabolic subgroup of $G_{k}$ corresponding to the simple root $\alpha_{k(k)}$. Let $W_{k}:=N_{G_{k}}(T_{k})/T_{k}$ denote the Weyl group of $G_{k}$ with respect to $T_{k}$. For a fixed  $1\leq k\leq r$, let $s_{\alpha_{i(k)}}$ denote the simple reflection in $W_{k}$ corresponding to simple root $\alpha_{i(k)}$ for $1\leq i\leq n_{k}-1$. Let $X(w_{k,n_{k}})\subseteq G_{k}/Q_{k}$ denote the  unique minimal dimensional Schubert variety admitting semistable points for $T_{k}$-linearized line bundle $\mathcal{L}(n_{k}\omega_{k(k)})$. From \cite[Lemma 2.7]{sskannan}, for $1\leq k\leq r$, a reduced expression for $w_{k,n_{k}}$ is given by 
	\begin{align*}
		w_{k,n_{k}}=(s_{\alpha_{q(k)}}s_{\alpha_{q-1(k)}}\cdots s_{\alpha_{1(k)}})(s_{\alpha_{2q(k)}}s_{\alpha_{2q-1(k)}}\cdots s_{\alpha_{2(k)}})\cdots(s_{\alpha_{kq(k)}}s_{\alpha_{kq-1(k)}}\cdots s_{\alpha_{k(k)}}).
	\end{align*}
	Let $T_{J_{k}}$ denote the subgroup of $T_{k}$ generated by $\lambda_{jq(k)}(\mathbb{G}_{m})$ ($1\leq j\leq k$). For $1\leq k\leq r$, let $Y_{k}:=\GmodX{T_{J_{k}}}{X(w_{k,n_{k}})}{\mathcal{L}(n_{k}\omega_{k(k)})}$.
	
	From \cref{section 5.1}, for each $2\leq k\leq r$, we have a morphism  $\phi_{k}: Y_{k}\longrightarrow Y_{k-1}$. Let $Y_{1}\xrightarrow{\text{$\phi_{1}$}}Y_{0}=\{pt\}$ denote the constant morphism.

	\begin{corollary}\label{section:projective bundle:cor 1} $Y_{r}\xrightarrow{\text{$\phi_{r}$}}Y_{r-1}\xrightarrow{\text{$\phi_{r-1}$}}\cdots Y_{2}\xrightarrow{\text{$\phi_{2}$}}Y_{1}\xrightarrow{\text{$\phi_{1}$}}Y_{0}=\{pt\}$
		is  a generalized Bott tower with stage $r$.
	\end{corollary}
	\begin{proof} 
		We have $w_{1,n_{1}}=s_{\alpha_{q(1)}}s_{\alpha_{q-1(1)}}\cdots s_{\alpha_{1(1)}}$. Thus from \cite[Lemma 5.8]{GKgit}, $X(w_{1,n_{1}})$ is the unique minimal dimensional Schubert variety in $G_{1}/Q_{1}$ admitting semistable points for the $\lambda_{q(1)}(\mathbb{G}_{m})$-linearized line bundle $\mathcal{L}(n_{1}\omega_{1(1)})$. Therefore, from \cite[Theorem 7.1]{GKgit}, we have $Y_{1}$ is isomorphic to $\mathbb{P}^{q-1}$.\\		
		Fix $2\leq k\leq r$. Then using \cref{section:projective bundle:th 1},  we have $Y_{k}$ is isomorphic to the projective space bundle $\mathbb{P}(\mathcal{E}_{k})$ over $Y_{k-1}$, where $\mathcal{E}_{k}=\mathcal{L}_{k,0}^{\oplus q}\oplus(\bigoplus_{j=1}^{k-1}\mathcal{L}^{\oplus q-1}_{k,j})$ and for all $0\leq j\leq k-1$, $\mathcal{L}_{k,j}$ is a line bundle on $Y_{k-1}$. 
		Now the proof follows from the definition of generalized Bott tower in \cref{subsection 2.2}.
	\end{proof}
\section{Acknowledgement}We thank Infosys foundation for partial financial support. The first named author would like to thank
the National Board for Higher Mathematics (NBHM), Department of
Atomic Energy, Government of India (Ref.No.0203/20/2022-R\&D-II/13472) for financial support.

		 \end{document}